\documentclass[12pt]{amsart}

\usepackage{latexsym,url}
\usepackage{amsthm}
\usepackage{amsmath}
\usepackage{amsfonts}
\usepackage{amssymb}
\usepackage[dvips]{graphicx}
\usepackage{xypic}

\input xy
\xyoption{all}

\newcommand{\+}{\nobreakdash-}
\newcommand{\tim}{\rightthreetimes}

\theoremstyle{plain}
\newtheorem{theo}{Theorem}[section]
\newtheorem{lemma}[theo]{Lemma}

\newtheorem{propo}[theo]{Proposition}
\newtheorem{coro}[theo]{Corollary}
\theoremstyle{definition}
\newtheorem{defi}[theo]{Definition}
\newtheorem{rem}[theo]{Remark}

\newtheorem{exam}[theo]{Example}
\newtheorem{exams}[theo]{Examples}

\newcommand\Mono{\operatorname{Mono}}
\newcommand\Epi{\operatorname{Epi}}
\newcommand\Mod{\operatorname{Mod}}
\newcommand\Discr{\operatorname{Discr}}
\newcommand\Contra{\operatorname{Contra}}
\newcommand\Separ{\operatorname{Separ}}
\newcommand\Flat{\operatorname{Flat}}
\newcommand\Cotors{\operatorname{Cotors}}
\newcommand\Vfl{\operatorname{Vfl}}
\newcommand\Ctadj{\operatorname{Ctadj}}

\newcommand{\lMono}{\text{-}\!\Mono}
\newcommand{\lEpi}{\text{-}\!\Epi}
\newcommand{\lMod}{\text{-}\!\Mod}
\newcommand{\Modr}{\Mod\!\text{-}}
\newcommand{\Discrd}{\Discr\!\text{-}}
\newcommand{\lContra}{\text{-}\!\Contra}
\newcommand{\lSepar}{\text{-}\!\Separ}
\newcommand{\lFlat}{\text{-}\!\Flat}
\newcommand{\lCotors}{\text{-}\!\Cotors}
\newcommand{\lVfl}{\text{-}\!\Vfl}
\newcommand{\lCtadj}{\text{-}\!\Ctadj}

\newcommand\op{\operatorname{op}}
\newcommand\id{\operatorname{id}}

\newcommand\Ext{\operatorname{Ext}}
\newcommand\Tor{\operatorname{Tor}}
\newcommand\Ctrtor{\operatorname{Ctrtor}}

\newcommand\filt{\operatorname{filt}}
\newcommand\cof{\operatorname{cof}}
\newcommand\cell{\operatorname{cell}}

\newcommand\Set{\operatorname{\bf Set}}
\newcommand\Ab{\operatorname{\bf Ab}}

\newcommand\colim{\operatorname{colim}}
\newcommand\coker{\operatorname{coker}}

\newcommand\integers{\mathbb {Z}}
\newcommand\rationals{\mathbb {Q}}

\newcommand\ca{\mathcal {A}}
\newcommand\cb{\mathcal {B}}
\newcommand\cc{\mathcal {C}}

\newcommand\cf{\mathcal {F}}

\newcommand\cg{\mathcal {G}}
\newcommand\ch{\mathcal {H}}
\newcommand\ci{\mathcal {I}}

\newcommand\ck{\mathcal {K}}
\newcommand\cl{\mathcal {L}}
\newcommand\cm{\mathcal {M}}
\newcommand\cn{\mathcal {N}}
\newcommand\crr{\mathcal {R}}
\newcommand\cs{\mathcal {S}}
\newcommand\ct{\mathcal {T}}
\newcommand\cp{\mathcal {P}}

\newcommand\cx{\mathcal {X}}

\newcommand\bt{\mathbb {T}}

\newcommand\fa{\mathfrak {A}}
\newcommand\fr{\mathfrak {R}}
\newcommand\fii{\mathfrak {I}}
\newcommand\fu{\mathfrak {U}}
\newcommand\fp{\mathfrak {P}}
\newcommand\fj{\mathfrak {J}}
\newcommand\ff{\mathfrak {F}}
\newcommand\fq{\mathfrak {Q}}
\newcommand\ft{\mathfrak {T}}
\newcommand\fs{\mathfrak {S}}
\newcommand\fv{\mathfrak {V}}
\newcommand\fg{\mathfrak {G}}
\newcommand\fk{\mathfrak {K}}

\newcommand{\Fun}{\operatorname{Fun}}
\newcommand{\Lex}{\operatorname{Lex}}
\newcommand{\Rex}{\operatorname{Rex}}
\newcommand{\Fadd}{\operatorname{Fadd}}
\newcommand{\Ex}{\operatorname{Ex}}

\newcommand{\CT}{\operatorname{CT}}
\newcommand{\PL}{\operatorname{PL}}

\date{April~21, 2017}
 
\begin{document}

\title[Cotorsion theories in locally presentable abelian categories]
{Covers, envelopes, and cotorsion theories \\ in locally presentable
abelian categories \\ and contramodule categories}

\author[L. Positselski and J. Rosick\'{y}]
{L. Positselski and J. Rosick\'{y}}

\thanks{Supported by the Grant Agency of the Czech Republic under
the grant P201/12/G028.  The first-named author's research in Israel
is supported by the ISF grant~\#\,446/15} 

\address{
\newline L. Positselski\newline
Department of Mathematics\newline
Faculty of Natural Sciences, University of Haifa\newline
Mount Carmel, Haifa 31905, Israel -- and --\newline
Institute for Information Transmission Problems\newline
Bolshoy Karetny per.~19, Moscow 127051, Russia -- and -- \newline
Laboratory of Algebraic Geometry\newline
National Research Univerity Higher School of Economics\newline
Vavilova 7, Moscow 117312, Russia\newline
positselski@yandex.ru\newline
\newline J. Rosick\'{y}\newline
Department of Mathematics and Statistics\newline
Masaryk University, Faculty of Sciences\newline
Kotl\'{a}\v{r}sk\'{a} 2, 611 37 Brno, Czech Republic\newline
rosicky@math.muni.cz
}

\begin{abstract}
We prove general results about completeness of cotorsion theories
and existence of covers and envelopes in locally presentable abelian
categories, extending the well-established theory for module
categories and Grothendieck categories.
These results are then applied to the categories of contramodules
over topological rings, which provide examples and counterexamples.
\end{abstract} 


\maketitle

\section{Introduction}

\subsection{{}}
An abelian category with enough projective objects is determined by
(can be recovered from) its full subcategory of projective objects.
An explicit construction providing the recovery procedure is described,
in the nonadditive context, in the paper~\cite{CV} (see also~\cite{TH}),
where it is called ``the exact completion of a weakly left exact
category'', and in the additive context, in
the paper~\cite[Lemma~1]{K}, where it is called ``the category of
coherent functors''.
In fact, the exact completion in the additive context is contained
already in~\cite{F}, as it is explained in~\cite{RV}.

Hence, in particular, any abelian category with arbitrary coproducts
and a single small (finitely generated, finitely presented) projective
generator $P$ is equivalent to the category of right modules over
the ring of endomorphisms of~$P$.
More generally, an abelian category with arbitrary coproducts and a set
of small projective generators is equivalent to the category of
contravariant additive functors from its full subcategory formed by
these generators into the category of abelian groups (right modules
over the big ring of morphisms between the generators).

Now let $\ck$ be an abelian category with a single, not necessarily
small projective generator~$P$; suppose that all coproducts of copies
of $P$ exist in~$\ck$.
Then, for the same reasons, the category $\ck$ is determined by
the functor $\bt$ assigning to a set $S$ the set $\ck(P,P^{(S)})$ of
all morphisms from $P$ to the $S$-indexed coproduct of copies of~$P$.
The functor $\bt$ is endowed with a natural structure of a monad on
the category of sets; the full subcategory of coproducts of copies of
$P$ in $\ck$ is the category of free algebras over this monad; and
the whole category $\ck$ is equivalent to the category of all algebras
over the monad~$\bt$~\cite{V0}.
The functor $\ck(P,-)$ assigns to an object of $\ck$ the underlying
set of the corresponding $\bt$-algebra.

Conversely, the category of algebras over a monad $\bt$ on the category
of sets is abelian whenever it is additive.
Given a monad $\bt:\Set\to\Set$, elements of the set $\bt(S)$
can be viewed as $S$-ary operations on the underlying sets of algebras
over $\bt$, as the datum of an element $t\in \bt(S)$ allows to assign to
any map of sets $f:S\to A$ from $S$ into a $\bt$-algebra $A$ with
the structure map $a:\bt(A)\to A$ the element $a(\bt(f)(t))\in A$.
The category of algebras over a monad $\bt$ is additive/abelian if and
only if there are elements (``operations'') $x+y\in \bt(\{x,y\})$, \
$-x\in \bt(\{x\})$, and $0\in \bt(\varnothing)$ in $\bt$ satisfying
the obvious equations of compatibility with each other and the equations
of commutativity with all the other operations in~$\bt$ \cite[\S10]{W}.
We will call such a monad $\bt:\Set\to\Set$ \textit{additive}.

Furthermore, let $\kappa$ be a regular cardinal.
The object $P\in\ck$ is called \textit{abstractly $\kappa$-small} if any
morphism $P\to P^{(S)}$ factors through the coproduct of copies of $P$
indexed by a subset in $S$ of cardinality smaller than~$\kappa$
(cf.~\cite{V}).
In this case the functor $\bt:\Set\to\Set$ preserves $\kappa$-filtered
colimits, the object $P$ is $\kappa$-presentable, and the category $\ck$
is locally $\kappa$-presentable~\cite{AR}.
Conversely, if the category $\ck$ is locally presentable, then
the object $P$ is $\kappa$-presentable for some~$\kappa$, and hence
abstractly $\kappa$-small.
Monads preserving filtered colimits on the category of sets correspond
to finitary algebraic theories~\cite[Section~4.6]{Bo}; the categories
of algebras over such monads are called the categories of models of
algebraic theories in~\cite{L,ARV}.
These are what are called the algebraic monads or ``generalized rings''
in~\cite{D}. 
The categories of algebras over monads on $\Set$ preserving
$\kappa$-filtered colimits were studied in~\cite{W}; we will call
them the \textit{categories of models of $\kappa$-ary algebraic
theories}.

\subsection{{}}
Suppose now that the natural map from the coproduct of any family of
projective objects in $\ck$ to their product is a monomorphism;
equivalently, it suffices that the natural map $P^{(S)}\to P^S$ from
the coproduct of any number of copies of the object $P$ to their
product be a monomorphism.
Then the set $\bt(S)=\ck(P,P^{(S)})$ of all morphisms $P\to P^{(S)}$
in $\ck$ maps injectively into the set $\ck(P,P^S)$.
The latter set is bijective to the $S$-indexed product of copies of
the set of all morphisms $P\to P$ in~$\ck$.
Denote the associative ring opposite to the ring $\ck(P,P)$ by~$R$.

So, for any set $S$, the set $\bt(S)$ is a subset in $R^S$.
The $S$-ary operation in (the underlying set of) a $\bt$-algebra $A$
corresponding to an element of $\bt(S)$ can be now interpreted as
an operation of infinite summation of $S$-indexed families of elements
of $A$ with the coefficients belonging to~$R$.
The $S$-indexed families of coefficients in $R$ for which such
an infinite summation operation exists in $\bt$ (i.~e.,
the $S$-indexed families of elements in $R$ belonging to the subset
$\bt(S)\subseteq R^S$) we will call \textit{$\bt$-admissible}.
Since we assume $\ck$ to be an abelian category ($\bt$ to be an additive
monad), we have $\bt(S)=R^S$ for any finite set $S$, that is any finite
family of elements in $R$ is $\bt$-admissible.

The datum of an abelian category $\ck$ with a projective generator $P$
such that arbitrary coproducts of copies of $P$ exist in $\ck$ and map
monomorphically into the respective products (or an additive monad
$\bt:\Set\to\Set$ such that $\bt(S)\subseteq\prod_{s\in S}\bt(\{s\})$
for all sets $S$) is equivalent to the following set of data:

\begin{itemize}
\item a associative ring $R$ with unit;
\item for any set $S$, a subset (in fact, an $R$-$R$-subbimodule)
of ``$\bt$-admissible families of coefficients'' $\bt(S)\subseteq R^S$;
and
\item the map (in fact, an $R$-$R$-bimodule morphism) of ``sum of
the coefficients in an admissible family'' $\sum_S: \bt(S)\to R$,
\end{itemize}

which has to satisfy the following conditions:

\begin{enumerate}
\renewcommand\theenumi{\roman{enumi}}
\item all the finite families of coefficients are admissible, i.~e.,
$\bt(S)=R^S$ for a finite set $S$; the map $\sum_S$ is
the conventional finite sum in the ring $R$ in this case;
\item any subfamily of an admissible family of coefficients is
admissible, i.~e., for any subset $S'\subseteq S$, the projection
map $R^S\to R^{S'}$ takes $\bt(S)$ into $\bt(S')$;
\item given a map of sets $U\to S$, for any family of
coefficients $t\in \bt(U)$ the family of sums
$S\ni s\mapsto \sum_{u\in f^{-1}(s)}t(u)$ belongs to $\bt(S)$, and
the iterated sum is equal to the sum of the original family,
$\sum_{s\in S}\sum_{u\in f^{-1}(s)}t(u)=\sum_{u\in U}t(u)$;
\item for any family of coefficients $r\in \bt(S)$ and any family of
families of coefficients $t_s\in \bt(f^{-1}(s))$ defined for all $s\in S$,
the family $U\ni u\mapsto r(f(u))t_{f(u)}(u)$ belongs to $\bt(U)$, and
$\sum_{s\in S} r(s)\sum_{u\in f^{-1}(s)}t_s(u)=
\sum_{u\in U}r(f(u))t_{f(u)}(u)$. 
\end{enumerate}

Given a monad $\bt:\Set\to\Set$, the maps $\sum_S:\bt(S)\to R=\bt(\{*\})$
are obtained by applying the functor $\bt$ to the projection maps
$S\to\{*\}$.
Given a ring $R$ with the subsets $\bt(S)\subseteq R^S$ and the maps
$\sum_S:\bt(S)\to R$, the functoriality maps $\bt(f):\bt(U)\to \bt(S)$
are defined in terms of the maps $\sum_S$, and the monad multiplication
maps $\bt(\bt(U))\to \bt(U)$ are defined in terms of the multiplication
operation in the ring $R$ together with the maps $\sum_S$, using 
the above axioms; while the monad unit maps $U\to \bt(U)$ are defined
in terms of the unit element of the ring~$R$.

In particular, it follows from the conditions~(i\+-iii) that no nonzero
element of $R$ can occur in an admissible family of coefficients
more than a finite number of times.
Hence, given a family of coefficients $t\in \bt(S)$, the cardinality
of the set of all $s\in S$ for which $t(s)\ne0$ cannot exceed
the cardinality of~$R$.
Thus any abelian category $\ck$ with a projective generator $P$ such
that the natural maps $P^{(S)}\to P^S$ are monomorphisms for all sets
$S$ is locally $\kappa$-presentable, where $\kappa$ is the successor
cardinal of the cardinality of the set $R=\ck(P,P)$.

A natural class of examples of monads $\bt$ or abelian categories $\ck$
as above is related to topological rings.
Let $\fr$ be a complete, separated topological ring where open right
ideals form a base of neighborhoods of zero.
For any set $S$, set $\bt(S)=\fr[[S]]$ to be the set of all formal linear
combinations of elements of $S$ with the coefficients in $\fr$ forming
a family converging to zero in the topology of $\fr$; in other words,
a family of coefficients $t\in\fr^S$ belongs to $\bt(S)$ if and only if
for any neighborhood of zero $\fu\subseteq\fr$ the set of all $s\in S$
for which $t(s)\notin\fu$ is finite.
The summation map $\sum_S:\fr[[S]]\to\fr$ is defined as the infinite sum
of the coefficients taken with respect to the topology of $\fr$, i.~e.,
the limit of finite partial sums in~$\fr$.
The abelian category of algebras/modules over the monad
$S\mapsto\fr[[S]]$ is called the category of \textit{left contramodules}
over topological ring~$\fr$.

Contramodules were introduced originally by Eilenberg and
Moore~\cite{EM} in the case of coalgebras over commutative rings
as module-like objects dual-analogous to, but different from,
the more familiar comodules.
They were subsequently studied by Barr~\cite{Ba}.
The definition of a contramodule was extended to topological rings
in~\cite[Appendices~A and~D]{P} and~\cite{P2}
(see~\cite{P4} for an overview).

We do not know whether all additive monads $\bt:\Set\to\Set$
such that $\bt(S)\subseteq\prod_{s\in S}\bt(\{s\})$ for all sets $S$
come from topological rings.
One can try to define a topology on the ring $R=\bt(\{*\})$ by calling
a subset $V\subseteq R$ a neighborhood of zero whenever for any
$t\in \bt(S)$ one has $t(s)\in V$ for all but a finite number of
elements $s\in S$; but it is not clear, e.~g., why does this topology
need to be separated.

Nevertheless, the categories of contramodules over topological rings
provide a large class of locally presentable abelian categories with
enough projective objects, which play a role in homological algebra,
commutative algebra, and representation theory.
Apparently, it was precisely the (perceived) lack of relevant examples
that discouraged the study of locally presentable abelian categories
in general, and abelian categories of models of $\kappa$-ary
algebraic theories in particular, for many years after
the monographs~\cite{W} and~\cite{AR} appeared in print.
The present work, motivated originally by the first-named author's
interest in the possibilities of applying contemporary set-theoretic
techniques of homological algebra and category theory to
the categories of contramodules, aims to begin filling this gap.

\subsection{{}}
Having discussed the place of contramodule categories in the general
category theory, let us now say a few words about covers, envelopes,
and cotorsion theories.
The classical \textit{flat cover conjecture} has two approaches to
its solution developed in the literature.
In fact, the paper~\cite{BBE}, where this conjecture was first proved,
contains two proofs.
The proof of Bican and El~Bashir, based on the approach suggested by
Bican and Torrecillas~\cite{BT}, was subsequently generalized by
El~Bashir in~\cite{B} to the claim that, for any accessible full
subcategory $\ca$ closed under coproducts and directed colimits in
a Grothendieck abelian category $\ck$, any object of $\ck$ has
an $\ca$-cover.
In this paper, we generalize this result even further by replacing
a Grothendieck category with a locally presentable
(or ``$\kappa$-Grothendieck'') abelian category.

The second proof of the flat cover conjecture, due to Enochs, was
based on the result of Eklof and Trlifaj~\cite{ET} claiming that any
cotorsion theory generated by a set of objects in the category of
modules over a ring (or, slightly more generally, in a Grothendieck
abelian category with enough projective objects) is complete.
As an accessible full subcategory $\ca$ closed under coproducts and
directed colimits does not have to be closed under extensions in $\ck$,
while a deconstructible class of objects containing the projectives
and closed under direct summands (the left part $\cf$ of
a cotorsion theory $(\cf,\cc)$ generated by a set $\cs$) does not
have to be closed under directed colimits, the approaches of
Bican--El~Bashir and Eklof--Trlifaj complement each other nicely,
neither of the two results being a particular case of the other one.

In fact, it is easy to observe that the Eklof--Trlifaj theorem does not
require projective objects and can be stated for arbitrary Grothendieck 
abelian categories $\ck$ as follows.

\begin{theo} \label{eklof-trlifaj}
 Let $\cs$ be a set of objects of $\ck$, and let $\cc=\cs^\perp$ and
$\cf={}^\perp\cc$ be the corresponding\/ $\Ext^1_{\ck}$-orthogonal
classes.
 Let\/ $\filt(\cs)$ denote the class of all transfinitely iterated
extensions, in the sense of the directed colimit, of objects from $\cs$
in~$\ck$.
 Suppose that every object of $\ck$ is a quotient object of
an object of\/ $\filt(\cs)$.
 Then $(\cf,\cc)$ is a complete cotorsion theory in $\ck$, that is
any object $K$ of $\ck$ can be included into short exact sequences
\begin{gather*}
0\longrightarrow K\longrightarrow C\longrightarrow F'\longrightarrow0 \\
0\longrightarrow C'\longrightarrow F\longrightarrow K\longrightarrow0
\end{gather*}
with $C$, $C'\in\cc$ and $F'$, $F\in\cf$.
 Furthermore, the class $\cf$ consists precisely of all the direct
summands of objects from\/ $\filt(\cs)$.
\end{theo}

The assertion of Theorem~\ref{eklof-trlifaj} in the form stated above
is \textit{not} true for locally presentable abelian categories.
The simplest contramodule categories with nonexact functors of filtered
colimits, such as the categories of contramodules over the topological
rings of $p$-adic integers $\integers_p$ or of the formal power series
$k[[z]]$ in one variable $z$ over a field~$k$, allow one to demonstrate
counterexamples by making an obvious ``bad'' choice of a generating
object or set of objects~$\cs$.
Nevertheless, our version of the Eklof--Trlifaj theorem applicable to
any locally presentable abelian category is surprisingly close to
the formulation above, the only essential difference is that one has
to explicitly require the object of $\ck$ to be a subobject of
an object from $\cc$ (in addition to it being a quotient object of
an object from $\cf$ or $\filt(\cs)$).
This condition holds automatically in Grothendieck categories, which
always have enough injective objects; but locally presentable
abelian categories may contain no injectives. 

An alternative approach is to work with generating sets $\cs$ or
classes $\cf$ consisting of objects which have been shown to behave
well with respect to filtered colimits.
This is particularly important for our construction of $\cc$-envelopes
in locally presentable abelian categories, which is based entirely on
the assumption of preservation of $\cf$-monomorphisms by directed
colimits, at least in the comma-categories $K\downarrow\ck$.
So we have to engage in a homological study of \textit{flat
contramodules} over topological rings in order to prove that
the derived functors of filtered colimits in contramodule categories
vanish on diagrams of flat contramodules.
In the end, our results on covers, envelopes, and cotorsion theories
in contramodule categories include
\begin{itemize}
\item existence of $\ca$-covers for any accessible full subcategory
$\ca$ closed under coproducts and directed colimits;
\item in particular, existence of flat covers;
\item completeness of the flat cotorsion theory, that is existence
of special flat precovers and special cotorsion preenvelopes;
\item existence of cotorsion envelopes;
\item completeness of any cotorsion theory $(\cf,\cc)$ generated
by a set $\cs$ of flat contramodules;
\item completeness of any cotorsion theory $(\cf,\cc)$ generated
by a set of contramodules $\cs$ such that every contramodule is
a subcontramodule of a contramodule from~$\cc$.
\end{itemize}

\subsection{{}}
 We discuss $\kappa$-Grothendieck abelian categories and develop
our version of the Bican--El~Bashir approach to covers
in Section~\ref{precovers-secn}.
 Our version of the Eklof--Trlifaj theorem is considered in
Section~\ref{cotorsion-and-envelopes}, with
further details following in Section~\ref{cotorsion-further}.
 Counterexamples showing that Theorem~\ref{eklof-trlifaj} does not
hold in locally presentable abelian categories are also demonstrated
in Section~\ref{cotorsion-further}, and envelopes are discussed
in Section~\ref{cotorsion-and-envelopes}.
 We recall the basic definitions and constructions related to
contramodules in Section~\ref{contramodule-basics}.
 Flat contramodules are studied in Section~\ref{flat-contramodules-secn}.
 Cotorsion theories in contramodule categories are constructed in
Section~\ref{cotorsion-in-contramodules}.

\subsection*{Acknowledgment}
 The authors are grateful to J.~Trlifaj and J.~\v St'ov\'\i\v cek
for helpful discussions.
 The first-named author wishes to express his gratitude to Masaryk
University in Brno, Charles University in Prague, and Eduard \v Cech
Institute for Algebra, Geometry and Physics for their hospitality,
which made it possible for this work to appear.

\section{$\kappa$-Grothendieck categories, Precovers and Covers}
\label{precovers-secn}

Recall that a \textit{set of generators} of a category $\ck$ is a set
$\cg$ of objects of $\ck$ such that any object $K$ of $\ck$ is a quotient
of a coproduct of objects from $\cg$. This is clearly equivalent to
the condition that for any two distinct morphisms $f_1,f_2:K\to L$ there
is a morphism $g:G\to K$ with $G\in\cg$ such that $f_1g\neq f_2g$.

\begin{defi}\label{def2.1} 
Let $\ck$ be a cocomplete abelian category with a set of generators
and $\kappa$ a regular cardinal. We say that $\ck$ is
\textit{$\kappa$-Grothendieck} if $\kappa$-filtered colimits
are exact in~$\ck$.
\end{defi}

Let $\ck$ be a category and $\lambda$ a regular cardinal.
An object $K$ in $\ck$ is called \textit{$\lambda$-presentable} if
the functor $\ck(K,-):\ck\to\Set$ preserves $\lambda$-filtered colimits.
A category $\ck$ is called \textit{$\lambda$-accessible} if 
$\lambda$-filtered colimits exist in $\ck$ and there is a set $\cg$ of
$\lambda$-presentable objects such that every object of $\ck$ is
a $\lambda$-filtered colimit of objects from~$\cg$.
A functor $F:\ck\to\cl$ is \textit{$\lambda$-accessible} if
$\ck$ and $\cl$ are $\lambda$-accessible categories and $F$ preserves
$\lambda$-filtered colimits~\cite{AR,MP}.

A category or functor is \textit{accessible} if it is
$\lambda$-accessible for some regular cardinal~$\lambda$.
A full subcategory $\ck\subseteq\cl$ is \textit{accessibly embedded} if
there exists a cardinal $\lambda$ such that $\ck$ is closed under
$\lambda$-filtered colimits in~$\cl$.
A category is \textit{locally $\lambda$-presentable} if it is
$\lambda$-accessible and cocomplete.
A category is \textit{locally presentable} if it is locally
$\lambda$-presentable for some regular cardinal~$\lambda$.

In a locally $\kappa$-presentable category, $\kappa$-filtered colimits
commute with $\kappa$-small limits \cite[Proposition~1.59]{AR},
so locally $\kappa$-presentable abelian categories are
$\kappa$-Grothendieck.

\begin{theo}\label{th2.2} Any $\kappa$-Grothendieck category
is locally presentable.
\end{theo}

\begin{proof}
Following~\cite[Remark~2.2]{V}, any cocomplete abelian category $\ck$
with a set of generators such that $\kappa$-filtered colimits commute
with $\nu$-small limits in $\ck$ for some regular cardinals
$\nu\leq\kappa$ is a $\nu$-localization of the category of models $\cm$
of a certain $\kappa$-ary algebraic theory.
(In fact, the category $\cm$ is also abelian.)
This means that $\ck$ is a reflective full subcategory of $\cm$
such that the left adjoint functor to the inclusion of $\ck$ to $\cm$
preserves $\nu$-small limits.
In particular, $\cm$ is locally $\kappa$-presentable, $\ck$ is
a localization of $\cm$ and, following~\cite[Proposition~6.7]{BK},
$\ck$ is closed under $\alpha$-filtered colimits in $\cm$ for some
regular cardinal $\alpha\geq\kappa$.
Thus $\ck$ is locally $\alpha$-presentable
(see~\cite[Remark~1.20 and Theorem~1.39]{AR}). 
\end{proof}

\begin{lemma}\label{lq1.1}
Let $\ck$ be a locally $\lambda$-presentable category, $\gamma$ be
the number of morphisms between (non-isomorphic) $\lambda$-presentable
objects of $\ck$ and $\kappa\trianglerighteq\lambda$, $\kappa>\gamma$
be a regular cardinal. Then a $\kappa$-presentable object of $\ck$
cannot have an increasing chain of subobjects of length~$\kappa$.
\end{lemma}

\begin{proof}
Let $\ck$ be a locally $\lambda$-presentable category and $\cx$ be its
representative full subcategory of $\lambda$-presentable objects.
Consider the canonical full embedding $E:\ck\to\Set^{\cx^{\op}}$ sending
$K$ to $\ck(-,K)$ restricted on $\cx$. This functor sends any increasing
chain of subobjects of $K$ to the increasing chain of subobjects of~$EK$.
Since any $\kappa$-presentable object of $\ck$ is a $\kappa$-small
$\lambda$-filtered colimit of $\lambda$-presentable objects for
$\kappa\trianglerighteq\lambda$ (see~\cite[Proposition~2.3.11]{MP}), 
and $E$ preserves $\lambda$-filtered colimits and sends 
$\lambda$-presentable objects to finitely presentable ones, $E$ sends 
$\kappa$-presentable objects to $\kappa$-presentable ones. 
The forgetful functor $\Set^{\cx^{\op}}\to\Set^X$, where $X$ is the set of
objects of $\cx$, preserves $\kappa$-presentable objects when
$\kappa>\gamma$.
This implies that a $\kappa$-presentable object in $\ck$ cannot have
an increasing chain of subobjects of length~$\kappa$.
In fact, such a chain would yield a similar chain of subobjects of~$EK$.
Applying the forgetful functor, we would get an increasing chain of
subsets of the disjoint union of $EK(A)$ over the objects $A$ of~$\cx$.
Since $EK$ is $\kappa$-presentable, all the sets $EK(A)$ have
cardinality~$<\kappa$ and thus this chain cannot have length~$\kappa$.
\end{proof}

Recall that a morphism $f:K\to L$ is a \textit{$\lambda$-pure
epimorphism} if any morphism $A\to L$ with $A$ $\lambda$-presentable
factorizes through~$f$. In a locally $\lambda$-presentable category
any $\lambda$-pure epimorphism is an epimorphism.

\begin{defi}\label{def2.3}
We say that a category $\ck$ \textit{has enough $\lambda$-pure quotients}
if for each object $K$ there is, up to isomorphism, only a set of
morphisms $f:L\to K$ such that $f=h\cdot g$, \ $g$ $\lambda$-pure
epimorphism implies that $g$ is an isomorphism.
\end{defi}

Clearly, if $\ck$ has enough $\lambda$-pure quotients then it has enough
$\mu$-pure quotients for any $\mu\geq\lambda$.

In an abelian locally $\lambda$-presentable category, $\lambda$-pure
epimorphisms are precisely cokernels of $\lambda$-pure monomorphisms
(see~\cite{AR1}). Recall that a monomorphism $f:K\to L$ is $\lambda$-pure
if in every commutative square
$$
\xymatrix{
K\ar [r]^{f}& L\\
A\ar [u]^{u}\ar [r]_{g}& B\ar [u]_{v}
}
$$
with $A$ and $B$ $\lambda$-presentable the morphism $u$ factors
through~$g$. 

Thus an abelian locally $\lambda$-presentable category has enough
$\lambda$-pure quotients if and only if for each object $K$ there is,
up to isomorphism,  only a set of morphisms $f:L\to K$ such that
$fg=0$ and $g$ $\lambda$-pure monomorphism implies that $g=0$.
Following~\cite[Theorem 2.1]{B}, any Grothendieck category has enough
$\lambda$-pure quotients for some regular cardinal~$\lambda$.
As observed in~\cite{R}, this is valid in any abelian locally
presentable category:

\begin{theo}\label{th2.4} Any additive locally presentable category has 
enough $\lambda$-pure quotients for some regular cardinal~$\lambda$.
\qed
\end{theo}

An argument deducing Theorem~\ref{th2.4} from its particular case when
the category in question is the category of modules over a ring with
several objects (hence a Grothendieck category) can be found
in~\cite[Corollary~1]{K2}.

Recall that a full subcategory $\ca$ of a category $\ck$ is called
\textit{weakly coreflective} if each object $K$ in $\ck$ has
a  weak coreflection, i.e., a morphism $c_K : K^\ast \to K$ where
$K^\ast$ is in $\ca$ such that every morphism $f:A \to K$ with $A$
in $\ca$ factorizes  (not necessarily uniquely) through~$c_K$. Every
weakly coreflective subcategory closed under retracts is closed
under coproducts in~$\ck$. The following result was proved
in~\cite[Proposition~2.3]{R}.

\begin{propo}\label{prop2.5}
Let $\ck$ be a locally presentable category having enough $\lambda$-pure 
quotients for some regular cardinal~$\lambda$. Let $\ca$ be
an accessible full subcategory of $\ck$ which is closed under coproducts
and directed colimits. Then $\ca$ is weakly coreflective in~$\ck$. \qed 
\end{propo}

In the module-theoretic terminology, $\ca$ being weakly coreflective
means that any object of $\ck$ has an $\ca$-precover. If, in addition,
any morphism $h:K^\ast\to K^\ast$ with $c_Kh=c_K$ is an isomorphism,
we say that $c_K:K^\ast\to K$ is an \textit{$\ca$-cover}.
The category-theoretic terminology is that $\ca$ is \textit{stably
weakly coreflective} (see~\cite{R}).

The following theorem is a state-of-the-art version
of~\cite[Theorem~3.1]{E}.

\begin{theo}\label{th2.6}
Let $\ck$ be a locally presentable category and $\ca$ its weakly
coreflective full subcategory which is closed under directed colimits.
Then $\ca$ is stably weakly coreflective.
\end{theo}

\begin{proof}
Let $\ck$ be a locally $\lambda$-presentable category.
Since $\lambda$-filtered colimits commute with finite limits,
the squaring functor $-^2:\ck\to\ck$ (sending $K$ to $K\times K$)
preserves $\lambda$-filtered colimits. Thus it preserves
$\mu$-presentable objects for some $\mu\geq\lambda$
(see~\cite[2.19]{AR}).
Let $K$ be in $\ck$ and assume that $K$ does not have a stable weak
coreflection to $\ca$. Let $c_K:K^\ast\to K$ be a weak coreflection.
There is a regular cardinal $\kappa$ such that
$\kappa\vartriangleright\lambda$, $\kappa\geq\mu$ and $\kappa>\gamma$
for which $K^\ast$ is $\kappa$-presentable (see~\cite[2.13(6)]{AR}).
Then the squaring functor preserves $\kappa$-presentable objects,
and a $\kappa$-presentable object cannot have an increasing chain
of subobjects of length~$\kappa$ (see Lemma~\ref{lq1.1}).

We will form the smooth chain $(t_{ij}:T_i\to T_j)_{i<j\leq\kappa}$
and a compatible cocone $c_i:T_i\to K$ starting with $T_0=K^\ast$ and
$c_0=c_K$. Having $c_i:T_i\to K$, we put $T_{i+1}=T_i$ and $c_{i+1}=c_i$.
If there is $f:T_i\to T_i$ with $c_if=c_i$ which is not a monomorphism,
we put $t_{i,i+1}=f$. If all morphisms $f:T_i\to T_i$ with $c_if=c_i$ are
monomorphisms, we take $t_{i,i+1}$ with $c_it_{i,i+1}=c_i$ which is not
an isomorphism. Since $c_i:T_i\to K$ is not a stable weak coreflection, 
such $t_{i,i+1}$ exists.
All objects $T_i$, $i<\kappa$ are $\kappa$-presentable.

Let us show that the ordinals $i<\kappa$ such that $t_{i\kappa}$ is
a monomorphism are cofinal in~$\kappa$. Consider kernel pairs
$s_j,r_j:S_j\to T_j\times T_j$ of $t_{j\kappa}$ and kernel pairs
$s_{ji},r_{ji}:S_{ji}\to T_j\times T_j$ of $t_{ji}$, $j<i$. Since
$\kappa$-directed colimits commute with finite limits, we get that
$S_j$ is the $\kappa$-directed colimit of $S_{ji}$, $j<i<\kappa$.
Since $T_j\times T_j$ does not contain an increasing chain of subobjects
of length $\kappa$, there is $j<i<\kappa$ such that $S_j=S_{ji}$.
We form a smooth sequence $i_0<i_1<\dots i_k<\dots$, $k\leq\lambda$
by putting $i_0=j$ and $S_{i_{k+1}}=S_{i_k,i_{k+1}}$. Since
$\lambda$-directed colimits commute with $\lambda$-small limits, we have
$$
S_{i_\lambda} = \colim_{k<\lambda} S_{i_k} = \colim_{k<\lambda}
S_{i_k,i_{k+1}}= T_{i_\lambda}.
$$
Hence $t_{i_\lambda,\kappa}$ is a monomorphism.

If $t_{i\kappa}$ is a monomorphism then $t_{ij}$ is a monomorphism for
all $i<j\leq\kappa$. In particular, $t_{i,i+1}$ is a monomorphism and,
following the construction, it is not an isomorphism. Moreover,
if $t_{ij}$ is an isomorphism for some $i<j\leq\kappa$, then $t_{i+1,j}$
is not a monomorphism and $t_{ij}^{-1}t_{i+1,j}$ is not a monomorphism,
which contradicts the condition that all morphisms $f:T_i\to T_i$
with $c_if=c_i$ are monomorphisms (we recall that $T_{i+1}=T_i$
and $c_{i+1}=c_i$). Thus $t_{ij}$ cannot be an isomorphism.
We have shown that $T_\kappa$ contains an increasing chain of subobjects
of length~$\kappa$.

Since $c_0: T_0\to K$ is a weak coreflection, there is a morphism
$h: T_\kappa \to T_0$ such that $c_0h=c_\kappa$. Once again, if
$t_{i\kappa}$ is a monomorphism then, following the construction,
all morphisms $f: T_i \to T_i$ such that $c_if=c_i$ are monomorphisms.
Hence $ht_{i\kappa}$ is a monomorphism because we have a monomorphism
$$
T_i \xrightarrow{\ ht_{i\kappa}\ } T_0 \xrightarrow{\ t_{0i}\ } T_i\,.
$$
Consequently, $h$ is a monomorphism and we get an increasing chain of 
subobjects of $T_0$ of length $\kappa$, which is a contradiction.
\end{proof}

\begin{coro}\label{cor2.7}
Let $\ck$ be an locally presentable abelian category and $\ca$ its
accessible full subcategory closed under coproducts and directed
colimits.  Then $\ca$ is stably weakly coreflective. \qed
\end{coro}

\begin{proof}
Combine Theorem~\ref{th2.4} and Proposition~\ref{prop2.5} with
Theorem~\ref{th2.6}.
\end{proof}

\begin{rem}\label{re2.8}
Assuming Vop\v enka's principle, any full subcategory of a locally
presentable category closed under directed colimits is accessible
(see~\cite[Theorem~6.17]{AR}). Thus any full subcategory of
an locally presentable abelian category closed under coproducts and
directed colimits is stably weakly coreflective.
\end{rem}

\section{Cotorsion theories, Special preenvelopes and Envelopes}
\label{cotorsion-and-envelopes}

 Given two classes of morphisms $\cl$ and $\crr$ in a category $\ck$,
we denote by $\cl^\square$ the class of all morphisms in $\ck$ having
the right lifting property with respect to every morphism from $\cl$,
and by  ${}^\square\crr$ the class of all morphisms having the left
lifting property with respect to every morphism from~$\crr$.
 We also denote by $\cl^\triangle$ the class of all objects in $\ck$
that are injective to every morphism from $\cl$, and by
${}^\triangle\crr$ the class of all objects that are projective
to every morphism in~$\crr$ (see~\cite[Definition~2.1]{R1}).

 A pair of classes of morphisms $(\cl,\crr)$ in a category $\ck$ is
called a \textit{weak factorization system} if $\crr=\cl^\square$,
$\cl={}^\square\crr$, and every morphism $f$ in $\ck$ can be
decomposed as $f=rl$ with $l\in\cl$ and $r\in\crr$
\cite[Definition~2.2]{R1}.
 The notion of a weak factorization system can be viewed as
``a half of'' a model structure: in a (closed) model category,
the pairs of classes of morphisms (cofibrations, trivial fibrations)
and (trivial cofibrations, fibrations) are weak factorization systems.

Let $\ck$ be an abelian category and $\ca$, $\cb$ its full subcategories.
Denote by $\ca^\perp$ the class of all objects $B\in\ck$ such that
$\Ext_{\ck}^1(A,B)=0$ for all $A\in\ca$.
Similarly, the class ${}^\perp\cb$ consists of all objects $A\in\ck$
such that $\Ext_{\ck}^1(A,B)=0$ for all $B\in\cb$.
The class $\ca^\perp$ is always closed under products in $\ck$ and
the class ${}^\perp\cb$ is closed under coproducts.

An \textit{$\ca$-monomorphism} is defined as a monomorphism in $\ck$
whose cokernel belongs to~$\ca$. The class of all $\ca$-monomorphisms
will be denoted by $\ca\lMono$.
Dually, the class of all epimorphisms whose kernel is in $\cb$ is
denoted as $\cb\lEpi$ and these epimorphisms are called
\textit{$\cb$-epimorphisms}.
Like in~\cite[Remark~4.3]{R1}, an object $B$ is injective to all
$\ca$-monomorphisms if and only if $\Ext^1_{\ck}(A,B)=0$ for all
$A\in\ca$. Thus $\ca\lMono^\triangle=\ca^\perp$ and
${}^\triangle\cb\lEpi={}^\perp\cb$.

\begin{lemma}\label{le3.1} Let $\ck$ be an abelian category and
$\ca$ a full subcategory of $\ck$ such that any object of $\ck$
is a quotient of an object from $\ca$. Then $g$ belongs to
$\ca\lMono^\square$ if and only if $g$ is an epimorphism whose kernel
is in $\ca^\perp$.
\end{lemma}

\begin{proof}
Let $g:C\to D$ belong to $\ca\lMono^\square$. Following the proof
of \cite[Lemma~4.4]{R1}, the kernel of $g$ belongs to $\ca^\perp$. 
There is $A\in\ca$ and an epimorphism $v:A\to D$. Consider the square
$$
\xymatrix{
0\ar [r] \ar [d] & C\ar [d]^{g}\\
A\ar [r]_{v}&D
}
$$
Since $0\to A$ is an $\ca$-monomorphism, there is $t:A\to C$ such that
$gt=v$. Thus $g$ is an epimorphism.  

The converse is~\cite[Lemma~1 in Section~9.1]{P}.
\end{proof}

 Given a class of morphisms $\cl$ in a cocomplete category $\ck$,
let $\cof(\cl)$ denote the closure of $\cl$ with respect to pushouts,
transfinite compositions, and retracts.
 For any class of morphisms $\crr$ in $\ck$, one has
$\cof({}^\square\crr)={}^\square\crr$.

\begin{lemma}\label{le3.2}
Let $\cb$ be a full subcategory of a cocomplete abelian category $\ck$
such that any object of $\ck$ is a subobject of an object from~$\cb$. 
Then $\cof({}^\perp\cb\lMono)={}^\perp\cb\lMono$.
\end{lemma}

\begin{proof}
Following the dual of Lemma~\ref{le3.1},
${}^\perp\cb\lMono$ = ${}^\square(\cb\lEpi)$. 
\end{proof}
 
Recall that a pair of classes of morphisms $(\cf,\cc)$ in an abelian 
category $\ck$ is called a \textit{cotorsion theory} if $\cc=\cf^\perp$
and $\cf={}^\perp\cc$. A cotorsion theory is called \textit{complete}
if any object $K$ has a \textit{special $\cf$-precover}, that is
a $\cc$-epimorphism $F\to K$ with $F\in\cf$, and a \textit{special
$\cc$-preenvelope}, i.~e., an $\cf$-monomorphism $K\to C$ with $C\in\cc$.

Since the objects of $\cf$ are projective with respect to
$\cc$-epimorphisms, any special $\cf$-precover of an object $K$
is an $\cf$-precover, that is a weak coreflection of $K$ to~$\cf$.
Dually, any special $\cc$-preenvelope is a $\cc$-preenvelope, i.~e.,
a weak reflection to~$\cc$ (see also~\cite[Propositions~2.1.3-4]{Xu}).

\begin{lemma}\label{le3.3} Let $(\cf,\cc)$ be a cotorsion theory in
an abelian category $\ck$ such that any object of $\ck$ is a quotient
of an object from $\cf$ and a subobject of an object from $\cc$.
Then $\cf\lMono^\square=\cc\lEpi$ and\/
${}^\square \cc\lEpi=\cf\lMono$.
\end{lemma}

\begin{proof}
It follows from Lemma~\ref{le3.1} and its dual.
\end{proof}

A cotorsion theory $(\cf,\cc)$ is \textit{generated} by a set $\cs$ if
$\cc=\cs^\perp$.

\begin{lemma} \label{presentable-codomains}
Let $\ck$ be a locally $\lambda$-presentable abelian category
and $\cs$ be a set of $\lambda$-presentable objects in~$\ck$.
Then any $\cs$-monomorphism in $\ck$ is a pushout of
an $\cs$-monomorphism with a $\lambda$-presentable codomain.
\end{lemma}

\begin{proof}
 Let $K\to L$ be a monomorphism with a cokernel $S\in\cs$.
 Let the object $L$ be the colimit of a $\lambda$-filtered diagram
of $\lambda$-presentable objects~$L_i$.
 Denote by $K_i$ the kernels of the compositions $L_i\to L\to S$.
 Since $\lambda$-filtered colimits commute with kernels in $\ck$,
 we have $K=\colim K_i$.
 Hence $S=\colim(L_i/K_i)$.
 Since the object $S$ is $\lambda$-presentable, we can conclude
that there exists~$i$ for which the morphism $L_i/K_i\to S$ is
a split epimorphism.
 Thus the morphism $L_i\to S$ is an epimorphism.
 We have constructed a commutative diagram
$$
\xymatrix{
0\ar[r] & K \ar[r] & L \ar[r] & S \ar[r] &0\\
0\ar[r] & K_i \ar[r]\ar[u] & L_i \ar[u]\ar[ur]
}
$$
where the two horizontal sequences are exact.
 It follows that the morphism $K\to L$ is a pushout of
the morphism $K_i\to L_i$.

 Indeed, let $K\to L'$ be the pushout of $K_i\to L_i$ by $K_i\to K$;
then there is the induced morphism $L'\to L$ and a commutative diagram
$$
\xymatrix{
0 \ar[r] & K \ar[r] \ar[dr] & L \ar[r] & S \ar[r] & 0 \\
& & L' \ar[u] \ar[ur]
}
$$
of a morphism between two short exact sequences with the same
subobject and the same quotient object.
A simple diagram chase now shows that both the kernel and
the cokernel of the morphism $L\to L'$ vanish.
Since the category $\ck$ is abelian, we can conclude that
$L\to L'$ is an isomorphism.
\end{proof}

\begin{propo}\label{prop3.4}
Let $\ck$ be a locally presentable abelian category and $(\cf,\cc)$
a cotorsion theory generated by a set. Suppose that any object of
$\ck$ is a quotient of an object from $\cf$ and a subobject of
an object from~$\cc$. Then $(\cf\lMono,\,\cc\lEpi)$ is a weak
factorization system.
\end{propo}

\begin{proof}
We can assume that $\ck$ is locally $\lambda$-presentable and any object
from a generating set $\cs_0$ of $(\cf,\cc)$ is $\lambda$-presentable. 
Any $\lambda$-presentable object $X$ is a quotient of an object $A_X$
from $\cf$. Add to the set $\cs_0$ the objects $A_X$ for $X$
$\lambda$-presentable. 
We get the set $\cs$ which generates $(\cf,\cc)$ and has the property
that any object of $\ck$ is a quotient object of a coproduct of 
objects from~$\cs$.

Let the set $\ci$ consist of all $\cs$-monomorphisms having
$\lambda$-presentable codomains.
We have $\ci^\square\supseteq\cf\lMono^\square=\cc\lEpi$.
Any $\cs_0$-monomorphism in $\ck$ is a pushout of a morphism from $\ci$
(see Lemma~\ref{presentable-codomains}), so
$\ci^\square\subseteq\cs_0\lMono^\square$.
Following the proof of Lemma~\ref{le3.1}, the kernel of any $g:C\to D$
from $\ci^\square$ belongs to~$\cc$.

Let us show that $g$ is an epimorphism. There is an epimorphism
$t:\coprod_{i\in I}X_i\to D$ where $X_i$, $i\in I$ are 
$\lambda$-presentable. There are epimorphisms $v_i:A_i\to X_i$ where
$A_i\in\cs$. Let $d_i:A_i\to C$ be diagonals in
$$
\xymatrix{
0\ar [r] \ar [d] & C\ar [d]^{g}\\
A_i\ar [r]_{t_iv_i}&D
}
$$
where $t_i:X_i\to D$ is the composition of $t$ with the coproduct
injection. Let $d:\coprod A_i\to C$ be induced by $d_i$. Then
$gd=tv$ and, since $v=\coprod_{i\in I}v_i$ is an epimorphism, $g$ is
an epimorphism. Hence $\ci^\square=\cc\lEpi$. Thus
$(\cf\lMono,\,\cc\lEpi)$ is a weak factorization system
(see~\cite[Proposition~1.3]{Be} or~\cite[Theorem~2.5]{R1}).
\end{proof}

\begin{coro}\label{cor3.5}
Let $\ck$ be a locally presentable abelian category and $(\cf,\cc)$
a cotorsion theory generated by a set. Suppose that any object of
$\ck$ is a quotient of an object from $\cf$ and a subobject of
an object from~$\cc$. Then $(\cf,\cc)$ is complete.
\end{coro}

\begin{coro}\label{cor3.6} Let $\ck$ be a locally presentable abelian
category and $(\cf,\cc)$ a complete cotorsion theory in $\ck$ such
that $\cf$ is closed under directed colimits. Then $\cf$ is stably
weakly coreflective in~$\ck$.
\end{coro}

\begin{proof}
This is a particular case of Theorem~\ref{th2.6}.
\end{proof}

\begin{rem}\label{re3.7}
Since any object has a special $\cf$-precover and an $\cf$-cover 
is a retract of any $\cf$-precover, an $\cf$-cover from
Corollary~\ref{cor3.6} is special (see also the proof of
Wakamatsu's lemma~\cite[Lemma~2.1.1]{Xu}).
\end{rem}

\begin{lemma}  \label{wfs-coreflection}
Let $(\cl,\crr)$ be a weak factorization system in a category $\ck$
and $K$ an object in~$\ck$.
Let $\cm$ be the full subcategory of the comma-category
$K\downarrow\ck$ consisting of the morphisms $K\to X$ belonging to~$\cl$.
Then $\cm$ is weakly coreflective in $K\downarrow\ck$.
\end{lemma}

\begin{proof}
 Decompose a morphism $f:K\to X$ as $f=rl$, where $l:K\to Z$ belongs
to $\cl$ and $r:Z\to X$ belongs to~$\crr$.
 Let $m:K\to Y$ be an object of $\cm$ and $t:Y\to X$ be a morphism
from $m$ to $f$ in~$K\downarrow\ck$.
 So we have a commutative square
$$
\xymatrix{
K\ar [r]^{l} \ar [d]_{m} & Z\ar [d]^{r}\\
Y\ar [r]_{t}&X
}
$$
 According to the lifting property of $m\in\cl$ with respect to
$r\in\crr$, there exists a morphism $h:Y\to Z$ in $\ck$ such that
$l=gm$ and $t=rh$.
 We have proved that $r:l\to f$ is a weak coreflection of $f$
to $\cm$ in $K\downarrow\ck$.
\end{proof}

We say that $\cf$ is \textit{strongly closed under directed colimits}
in $\ck$ if $\cf$-monomorphisms are closed under directed colimits in
any comma-category $K\downarrow\ck$. It implies that $\cf$ is closed
under directed colimits. In a Grothendieck category, the two concepts
are equivalent.

The following corollary is a version of~\cite[Theorems~2.2.6
and~3.4.3\+-4]{Xu} that works in locally presentable abelian categories.

\begin{coro}\label{cor3.9} Let $\ck$ be a locally presentable abelian
category and $(\cf,\cc)$ a complete cotorsion theory in $\ck$
such that $\cf$ is strongly closed under directed colimits.
Then $\cc$ is stably weakly reflective in~$\ck$.
\end{coro}

\begin{proof}
Following~\cite[Proposition~3.5]{R}, it suffices to show that for every
object $K$ in $\ck$ the morphism $K\to0$ has a stable weak coreflection
to the full subcategory $\cm$ of $K\downarrow\ck$ consisting of
$\cf$-monomorphisms.  According to Lemma~\ref{wfs-coreflection}, any
special $\cc$-preenvelope of $K$ is a weak coreflection of $K\to0$
to~$\cm$.  Following Theorem~\ref{th2.6}, $K\to0$ has a stable weak
coreflection to~$\cm$.
\end{proof}

\begin{rem}\label{re3.10}
Since any object has a special $\cc$-preenvelope and a $\cc$-envelope 
is a retract of any $\cc$-preenvelope, a $\cc$-envelope from
Corollary~\ref{cor3.9} is special (see also the proof
of~\cite[Lemma~2.1.2]{Xu}). 
\end{rem}

\begin{rem}
Following~\cite[Proposition~5.4]{H}, $(\cf\lMono,\,\cc\lEpi)$ is
a weak factorization system for any complete cotorsion theory
$(\cf,\cc)$ in an abelian category~$\ck$.
So, for any object $K$ in $\ck$ the subcategory $\cm$ is weakly
coreflective in $K\downarrow\ck$ by Lemma~\ref{wfs-coreflection}.
When $\ck$ is locally presentable and $\cf$ is strongly closed under
directed colimits, Theorem~\ref{th2.6} allows to conclude that $\cm$
is stably weakly coreflective.
Applying~\cite[Proposition~3.5]{R}, one proves, in the assumptions
of Corollary~\ref{cor3.9}, that $(\cf\lMono,\,\cc\lEpi)$ is
a stable weak factorization system in~$\ck$.
\end{rem}

\section{Further details on Cotorsion theories}
\label{cotorsion-further}

 The following counterexamples show that Theorem~\ref{eklof-trlifaj}
does not hold for locally presentable/$\kappa$-Grothendieck abelian
categories in general (see Corollary~\ref{cor3.5},
Theorem~\ref{precise-et-loc-pres}, or Proposition~\ref{et-sc-dc} for
formulations applicable to such categories).

\begin{exams} \label{ext-0-1-orthogonal-category}
(1) Let $\ck$ be an abelian category and $\ct\subseteq\ck$ be a class of
objects of projective dimension at most~$1$, that is $\Ext^n_\ck(T,K)=0$
for all objects $T\in\ct$ and $K\in\ck$ and all $n\ge2$.
Denote by $\cp=\ct^{\perp_{0,1}}\subseteq\ck$ the full subcategory formed
by all objects $P$ in $\ck$ such that $\Ext^n_\ck(T,P)=0$ for all
$T\in\ct$ and $n\ge0$.

Clearly, $\cp$ is closed under products and extensions in~$\ck$.
Now let $f:P\to Q$ be a morphism in $\cp$; denote its image by~$Z$.
Then for any $T\in\ct$ one has $\ck(T,Z)=0$, since $Z$ is
a subobject of $Q$, and $\Ext^1_\ck(T,Z)=0$, since $Z$ is
a quotient of~$P$.
Thus $Z$ is in $\cp$, and it follows that the kernel and cokernel
of~$f$ also belong to~$\cp$.
We have proven that $\cp$ is an abelian category and the inclusion
$\cp\to\ck$ is an exact functor.

\medskip

(2) Assume that $\ck$ is a locally $\lambda$-presentable abelian
category and $\ct$ is a set of $\lambda$-presentable objects.
 Let $\mu$ be a regular cardinal such that all subobjects of
$\lambda$-presentable objects in $\ck$ are $\mu$-presentable.
 Then one can see from Lemma~\ref{presentable-codomains} that
$\cp$ is closed under $\mu$-filtered colimits in~$\ck$.
 Following~\cite[Theorem and Corollary~2.48]{AR}, the category
$\cp$ is locally $\mu$-presentable and reflective in~$\ck$.

\medskip

(3) Let $R$ be a commutative ring, $\ck=R\lMod$ the category of
$R$-modules, $I\subset R$ an ideal, and $r_i\in I$ a set of generators
of~$I$.
 Let the set of $R$-modules $\ct$ consist of the localizations
$R[r_i^{-1}]$ of the ring $R$ with respect to its elements~$r_i$.
 Then the subcategory $\cp=\ct^{\perp_{0,1}}\subset R\lMod$ is a locally
$\aleph_1$-presentable abelian category with enough projective objects.
 In fact, its projective generator $P$ can be obtained as
the reflection of the free $R$-module $R\in\ck$ to~$\cp$.

 When the ideal $I$ is generated by a single nonzero-dividing element
$r\in R$, the functor $\ck\to\cp$ left adjoint to the inclusion can be
computed as $\Ext^1_R(R[r^{-1}]/R,-)$.
 When $R$ is a Noetherian ring, or $I=(r)$ with $r$ a nonzero-divisor,
the $R$-module $P$ can be computed as the $I$-adic completion
$\lim_m R/I^m$ of~$R$.
 The full subcategory $\cp\subseteq\ck$ does not depend on the choice of
a set of generators $r_i$ of the ideal~$I$, but only on the ideal itself.
 It is called the category of \textit{$I$-contramodule $R$-modules}
in~\cite{P5}.

\medskip

(4) Now let $I=(r)$ be the principal ideal generated by a noninvertible
nonzero-dividing element $r$ and $\ct$ be the set consisting
of a single $R$-module $T=R[r^{-1}]$.
 Let the set of objects $\cs$ consist of the projective generator
$P\in\cp$ and the $R$-module $S=R/rR$; since $r$ acts by zero in $S$ and
invertibly in $T$, one has $\Ext^n_R(T,S)=0$ for all $n\ge0$,
so $S\in\cp$.
 Let $\cf$, $\cc\subseteq\cp$ be the corresponding
$\Ext^1$-orthogonal classes, $\cc=\cs^\perp$ and $\cf={}^\perp\cc$,
in the abelian category~$\cp$.

 Note that the functors $\Ext^1_R$ and $\Ext^1_\cp$ agree on $\cp$,
as $\cp$ is closed under extensions in $R\lMod$.
 For any $R$-module $M$, one has $\Ext^1_R(S,M)=M/rM$.
 So $\Ext^1_R(S,M)=0$ means $M=rM$, which implies surjectivity of
the natural map $\ck(T,M)\to M$.
 For $M\in\cp$, this is only possible when $M=0$.
 We have shown that $\cc=0$ (in particular, there are no nonzero
injective objects in~$\cp$).
 Hence $\cf=\cp$.

 Consequently, nonzero objects of $\cp$ (such as, e.~g., $S$ or
$P=\lim_m R/r^mR$) cannot have special $\cc$-preenvelopes and
the assertion of Theorem~\ref{eklof-trlifaj} fails for
the locally $\aleph_1$-presentable abelian category $\cp$ with
the set of objects $\cs=\{S,P\}$.
\end{exams}

 The following lemma is a more precise formulation of the assertions
contained in (the proof of) Lemma~\ref{le3.1}.

\begin{lemma} \label{perp-and-square}
 Let $\ck$ be an abelian category and $\ca$ a class of objects
in~$\ck$. Then \par
\textup{(a)} the class of morphisms $\ca^\perp\lEpi$ is contained in
$\ca\lMono^\square$; \par
\textup{(b)} the kernel of any morphism from $\ca\lMono^\square$
belongs to $\ca^\perp$; \par
\textup{(c)} any morphism $X\to Y$ in $\ck$ which has right lifting
property with respect to all morphisms $0\to A$, \ $A\in\ca$ and
whose codomain $Y$ is a quotient of an object from $\ca$ is
an epimorphism. \qed
\end{lemma}

\begin{defi} \label{t.i.e}
Let $\alpha$ be an ordinal and $S_i$ a family of objects in a cocomplete
abelian category $\ck$ indexed by the ordinals $i<\alpha$.
We say that an object $F$ in $\ck$ is a \textit{transfinitely iterated
extension} of the objects $S_i$ (in the sense of the directed colimit)
if there exists a smooth chain $(f_{ij}:F_i\to F_j)_{i<j\le\alpha}$ in
$\ck$ such that $F_0=0$, \ $F_\alpha=F$, and $f_{i,i+1}$ is a monomorphism
with the cokernel isomorphic to $S_i$ for every $i<\alpha$.
\end{defi}

 Note that no exactness condition on the directed colimits is imposed in
Definition~\ref{t.i.e}.
 If $F$ is a transfinitely iterated extension of $S_i$, $i<\alpha$
in $\ck$ and an ordinal $\beta<\alpha$ is given, then one can claim that
$F_\beta$ is a transfinitely iterated extension of $S_j$, $j<\beta$ and
the cokernel $F_{\geq\beta}$ of the morphism $f_{\beta\alpha}$ is 
a transfinitely iterated extension of $S_k$, \ 
$\beta\leq k<\alpha$, so there is an exact sequence
$$
F_\beta \xrightarrow{\ f_{\beta\alpha}\ } F \xrightarrow{\ \ \ \ }
F_{\geq\beta}\xrightarrow{\ \ \ \ }0.
$$
 Furthermore, $F$ is a transfinitely iterated extension of the sequence
of objects $F_{\beta}$, $S_\beta$, $S_{\beta+1}$,~\dots, $S_i$,~\dots,
where $i<\alpha$.
 But the morphism $f_{\beta\alpha}$ does not need to be a monomorphism, 
as the following examples illustrate.

\begin{exams} \label{vanishing-t.i.e}
 Let $R$ be a commutative ring and $r\in R$ be a nonzero-dividing
element; set $S=R/rR$.
 Then in the abelian category $\cp$ from
Examples~\ref{ext-0-1-orthogonal-category}(3\+-4) the zero object is
a transfinitely iterated extension of the sequence of objects
$S$, $S$, $S$,~\dots{} indexed by the nonnegative integers~$n$.
 Indeed, set $F_n=R/r^nR\in\cp$, and let $f_{ij}:F_i\to F_j$ be
the monomorphism taking $1+r^iR$ to $r^{j-i}+r^jR$.
 Clearly, $\coker(f_{i,i+1})=S$.
 We claim that $\colim_n F_n=0$ in~$\cp$.

 Indeed, let $f_n:F_n\to F$ be a compatible cocone in $\cp$; then it
can be also viewed as a compatible cocone in $R\lMod$.
 The colimit of the diagram $(f_{ij}:F_i\to F_j)$ in the category
of $R$-modules is $R[r^{-1}]/R$, so the morphisms~$f_n$ factor
through an $R$-module morphism $f:R[r^{-1}]/R\to F$.
 Now any morphism $R[r^{-1}]\to F$ in $R\lMod$ vanishes by
the definition of the category $\cp$, hence $f=0$ and
consequently $f_n=0$ for all~$n$.

 Similarly one can show that the zero object is also a transfinitely
iterated extension of the sequence of objects $P$, $S$, $S$,
$S$,~\dots{} in~$\cp$, where $P=\lim_mR/r^mR$.
 It suffices to set $F_n=P$ and take $f_{ij}$ to be the morphisms of
multiplication with~$r^{j-i}$.
 Then the colimit of the diagram $(f_{ij}:F_i\to F_j)$ in $R\lMod$ is
the $R[r^{-1}]$-module $R[r^{-1}]\otimes_RP$, whose reflection to
$\cp$ vanishes, so $\colim_n F_n=0$ in~$\cp$.
\end{exams}

 These counterexamples explain why Eklof and Trlifaj's proof of their
theorem~\cite[Theorems~2 and~10]{ET} fails to work in locally
presentable abelian categories in general.
 The following lemma says, on the other hand, that the assertion
of Eklof's lemma~\cite[Lemma~1]{ET} is valid for any transfinitely
iterated extension of objects in a cocomplete abelian category.

\begin{lemma} \label{eklof-lemma}
 Let $\ck$ be a cocomplete abelian category and $\cb$ a class of objects
in~$\ck$.  Then the class of objects ${}^\perp\cb\subseteq\ck$ is
closed under transfinitely iterated extensions.
\end{lemma}

\begin{proof}
 This can be easily obtained from the lifting property considerations,
but we will outline a sketch of explicit proof directly comparable to
the argument in~\cite{ET}.
Let $F$ be a transfinitely iterated extension of a family of objects
$S_i\in{}^\perp\cb$, \ $i<\alpha$, and let
$$
 0\xrightarrow{\ \ \ } B\xrightarrow{\ \ \ } K
 \xrightarrow{\ p \ } F\xrightarrow{\ \ \ } 0
$$
be a short exact sequence in $\ck$ with $B\in\cb$.
 In order to show that this exact sequence splits, we proceed by
transfinite induction constructing a compatible cocone
$k_i:F_i\to K$ such that $pk_i=f_{i\alpha}$ for all $i\leq\alpha$.
 On a limit step~$j$, we simply set $k_j=\colim_{i<j}k_i$.
 On a successor step $j=i+1$, the pullback of our short exact sequence
with respect to the morphism $f_{i+1,\alpha}:F_{i+1}\to F$ defines
an extension class in $\Ext_\ck^1(F_{i+1},B)$.
 The argument continues exactly as in~\cite[Lemma~1]{ET}, proving that
this extension class vanishes, that there exists a morphism
$k'_{i+1}:F_{i+1}\to K$ for which $pk'_{i+1}=f_{i+1,\alpha}$, and finally
that the morphism $k'_{i+1}$ can be replaced with another morphism
$k_{i+1}:F_{i+1}\to K$ for which $pk_{i+1}=f_{i+1,\alpha}$ and
$k_{i+1}f_{i,i+1}=k_i$.
\end{proof}

Given a class of objects $\cs$ in a cocomplete abelian category $\ck$,
we denote by $\filt(\cs)$ the class of all transfinitely iterated
extensions of objects from~$\cs$.
Given a class of morphisms $\cl$ in $\ck$, we denote by $\cell(\cl)$
the closure of $\cl$ with respect to pushouts and transfinite
compositions, or, which is the same, the class of all transfinite
compositions of pushouts of morphisms from~$\cl$.
Clearly, any pushout of an $\cs$-monomorphism is
an $\cs$-monomorphism; so the class $\cell(\cs\lMono)$ consists
precisely of the transfinite compositions of $\cs$-monomorphisms.

The next lemma is our version of~\cite[Lemma~2.10 and
Proposition~2.11]{SS} (see also~\cite[Appendix~A]{S}).

\begin{lemma} \label{filt(s)-mono}
 Let $\ck$ be a cocomplete abelian category and $\cs$ a class of objects
in~$\ck$. Then \par
\textup{(a)} any\/ $\filt(\cs)$-monomorphism in $\ck$ belongs to
$\cell(\cs\lMono)$; \par
\textup{(b)} the cokernel of any morphism from $\cell(\cs\lMono)$
belongs to $\filt(\cs)$; \par
\textup{(c)} any morphism from $\cof(\cs\lMono)$ whose domain is
a subobject of an object from $\cs^\perp$ is a monomorphism; \par
\textup{(d)} the class of objects\/ $\filt(\filt(\cs))$ coincides with\/
$\filt(\cs)$.
\end{lemma}

\begin{proof}
 Part~(a): let $A\to B$ be a monomorphism whose cokernel $F$ is
a transfinitely iterated extension of objects $S_i\in\cs$, \ $i<\alpha$.
 Let $B_i\to F_i$ be the pullback of the epimorphism $B\to F$ with
respect to the morphism $f_{i\alpha}:F_i\to F$.
 Clearly, we have a commutative diagram
$(b_{ij}:B_i\to B_j)_{i<j\le\alpha}$, the morphism $b_{i,i+1}$ is
a monomorphism with the cokernel $S_i$, and it only remains to show that
$(b_{ij})$ is a smooth chain, that is $B_j=\colim_{i<j}B_i$ for all limit
ordinals $j\leq\alpha$.
 For every~$i$, we have a commutative diagram of a morphism of short
exact sequences with the same subobject
$$
\xymatrix{
0 \ar[r] & A \ar[r] \ar[dr] & B \ar[r] & F \ar[r] & 0 \\
& & B_i \ar[r] \ar[u] & F_i \ar[r] \ar[u] & 0
}
$$
Passing to the colimit, we conclude that the morphism
$\colim_{i<j}B_i\to\colim_{i<j}F_i=F_j$ is the cokernel of the morphism
$A\to\colim_{i<j}B_i$, since colimits commute with cokernels.
Furthermore, $A\to\colim_{i<j}B_i$ is a monomorphism,
because the composition $A\to\colim_{i<j}B_i\to B$ is.
Hence we obtain a commutative diagram of a morphism between two short
exact sequences with same subobject and quotient object
$$
\xymatrix{
0 \ar[r] & A \ar[r] \ar[dr] & \colim_{i<j}B_i \ar[r] \ar[d]
& F_j \ar[r] & 0 \\ & & B_j \ar[ur] 
}
$$
implying that $\colim_{i<j}B_i\to B_j$ is an isomorphism.

Part~(b) is a particular case of the assertion that a pushout of
a transfinite composition is a transfinite composition of pushouts,
following from the fact that colimits commute with pushouts.
To prove~(d), one notices that, according to part~(a),
$$
 \filt(\filt(\cs))\lMono \subseteq \cell(\filt(\cs)\lMono)
 \subseteq\cell(\cell(\cs\lMono)) = \cell(\cs\lMono),
$$
so the inclusion $\filt(\filt(\cs)\subseteq\filt(\cs)$ follows
from~(b).
Finally, part~(c) is implied by the dual of
Lemma~\ref{perp-and-square}(c), since morphisms from
$\cof(\cs\lMono)$ have the left lifting property with respect
to morphisms from $\cs^\perp\lEpi$ by Lemma~\ref{perp-and-square}(a).
\end{proof}

\begin{rem}
 Starting from Section~\ref{cotorsion-and-envelopes}, we always
assumed the categories where we considered transfinite compositions
of morphisms or transfinitely iterated extensions of objects to be
cocomplete, just to help the reader feel more confident; but this
assumption was never actually used until now.
 It was always sufficient to say instead that the classes
$\cof(\cl)$, $\cell(\cl)$, and $\filt(\cs)$ are formed using all
the colimits that happen to exist in the category~$\ck$.

 The above proof of Lemma~\ref{filt(s)-mono}(a), however, presumes
existence of $\colim_{i<j}B_i$.
 This can be avoided, too.
 Here is an alternative argument proving directly that
$(b_{ij}:B_i\to B_j)_{i<j}$ is a colimit cocone without any prior
assumption of the existence of the colimit.
 We have to show that for any object $X$ in $\ck$ the induced morphism
of abelian groups $c_X:\ck(B_j,X)\to\lim_{i<j}\ck(B_i,X)$ is
an isomorphism.
 We have the commutative diagram of a morphism of exact sequences of
abelian groups
$$
\xymatrix{
0\ar[r] & \lim_{i<j}\ck(F_j,X) \ar[r] &\lim_{i<j}\ck(B_i,X) \ar[r]
& \ck(A,X) \\ 0\ar[r] & \ck(F_j,X) \ar[r]\ar[u] & \ck(B_j,X)
\ar[r]\ar[u] & \ck(A,X) \ar[u]
}
$$
 The leftmost and the rightmost vertical maps are isomorphisms, hence
it follows from the diagram that the middle vertical map $c_X$ is
injective and a compatible cocone $(x_i:B_i\to X)_{i<j}$ belongs to
the image of $c_X$, i.~e., factors through the cocone $(b_{ij})_{i<j}$,
if and only if the morphism $x_0:A\to X$ factors through $A\to B_j$.

 Now let $h:X\to Y$ be the pushout of the monomorphism $A\to B_j$ by
the morphism~$x_0$.
 Then the composition $hx_0:A\to Y$ factors through $A\to B_j$, hence
the cocone $(hx_i:B_i\to Y)_{i<j}$ belongs to the image of~$c_Y$.
 We have obtained a morphism $y:B_j\to Y$ forming a commutative
diagram with the morphisms $hx_i$ and~$b_{ij}$.
 Let $p:Y\to Z$ be the cokernel of the monomorphism~$h$.
 Then $c_Z(py)=(phx_i)_{i<j}$ is the zero cocone, and since the morphism
$c_Z$ is injective, it follows that $py=0$.
 Thus there is a morphism $x:B_j\to X$ for which $y=hx$, and
one has $xb_{ij}=x_i$, since $hxb_{ij}=yb_{ij}=hx_i$.
 We have proven that the map $c_X$ is also surjective.
\end{rem}

 The following theorem is a more precise formulation of 
the Eklof--Trlifaj theorem for locally presentable abelian
categories (cf.\ Corollary~\ref{cor3.5}; see also
Proposition~\ref{et-sc-dc}).

\begin{theo} \label{precise-et-loc-pres}
 Let $\ck$ be a locally presentable abelian category, $\cs$ a set
of objects in $\ck$, and $(\cf,\cc)$ the cotorsion theory in $\ck$
generated by~$\cs$.  In this situation: \par
\textup{(a)} the class of objects\/ $\filt(\cs)$ is contained
in $\cf$; \par
\textup{(b)} if an object $X$ in $\ck$ is a subobject of an object $B$
from the class $\cc$, then $X$ can be included into a short exact
sequence $0\to X\to C\to F'\to0$ in $\ck$ with $C\in\cc$ and
$F'\in\filt(\cs)$;
\par
\textup{(c)} if an object $Y$ in $\ck$ is a quotient of an object $P$
from\/ $\filt(\cs)$, then $Y$ can be included into a short exact
sequence $0\to C'\to F\to Y\to 0$ in $\ck$ with $C'\in\cc$
and $F\in\filt(\cs)$; \par
\textup{(d)} if an object $G$ from the class $\cf$ is a quotient object
of an object from\/ $\filt(\cs)$, then $G$ is a direct summand of
an object from\/ $\filt(\cs)$.
\end{theo}

\begin{proof}
 Part~(a) is Lemma~\ref{eklof-lemma}.
 To prove parts~(b\+-d), choose a regular cardinal $\lambda$ such that
the category $\ck$ is locally $\lambda$-presentable and all objects
in the set $\cs$ are $\lambda$-presentable.
 Let $\ci$ be the set of all $\cs$-monomorphisms with
$\lambda$-presentable codomains in~$\ck$.
 By Lemmas~\ref{presentable-codomains} and~\ref{filt(s)-mono}(a),
we have $\ci^\square=\cs\lMono^\square=\filt(\cs)\lMono^\square$.

 According to Lemma~\ref{perp-and-square}, the class
$\cs\lMono^\square$ consists of morphisms with kernels in $\cc$,
contains all the $\cc$-epimorphisms, and any morphism
from $\cs\lMono^\square$ whose codomain is a quotient of
an object from $\filt(\cs)$ is an epimorphism.
 According to the dual of Lemma~\ref{perp-and-square}, the class
${}^\square(\cs\lMono^\square)$ consists of morphisms with cokernels
in $\cf$ and any morphism from ${}^\square(\cs\lMono^\square)$ whose
domain is a subobject of an object from $\cc$ is a monomorphism.
 Moreover, by Lemma~\ref{filt(s)-mono}(b) the class of morphisms
$\cell(\cs\lMono)\subset{}^\square(\cs\lMono^\square)$ consists
of morphisms with cokernels in $\filt(\cs)$.

 By the small object argument (see~\cite[Proposition~1.3]{Be}
or~\cite[Theorem~2.5]{R1}), any morphism in $\ck$ can be factored
as a composition of a morphism from $\cell(\cs\lMono)$ followed
by a morphism from $\cs\lMono^\square$.
 This implies the assertions~(b\+-c), and part~(d) follows from~(c).
\end{proof}

A class of objects $\cf$ in a cocomplete abelian category $\ck$ is
called \textit{deconstructible} if $\cf=\filt(\cs)$ for some set
of objects $\cs$ in~$\ck$.
The next corollary is our version of~\cite[Theorem~2.13(1) and
Corollary~2.15(1)]{SS}.

\begin{coro}
Any deconstructible class of objects in a locally presentable abelian
category is weakly coreflective.
\end{coro}

\begin{proof}
 For any class of objects $\ca$ in an abelian category $\ck$, any
morphism belonging to $\ca\lMono^\square$ with the domain belonging
to $\ca$ is an $\ca$-precover.
 In particular, let $\ck$ be locally presentable and $\cs$ a set of
objects in~$\ck$.
 Following the proof of Theorem~\ref{precise-et-loc-pres}, for any
object $K$ in $\ck$ the morphism $0\to K$ can be decomposed as
$0\to F\to K$ with the object $F$ from $\filt(\cs)$ and the morphism
$F\to K$ from $\cs\lMono^\square=\filt(\cs)\lMono^\square$.
 Hence $F\to K$ is a $\filt(\cs)$-precover of~$K$.
\end{proof}

\begin{rem}
 According to~\cite[Example~4]{K2}, for any cotorsion theory $(\cf,\cc)$
generated by a set of objects in a locally presentable abelian category
$\ck$, any object $X$ in $\ck$ has a $\cc$-preenvelope.
 If one of such preenvelopes is a monomorphism (or, equivalently, all
of them are), then by Theorem~\ref{precise-et-loc-pres}(b) \
$X$ has a special $\cc$-preenvelope.
\end{rem}

 Assuming Vop\v enka's principle, any full subcategory closed
under products in a locally presentable category is weakly
reflective (see~\cite[Theorem~6.26]{AR}).

\begin{coro} \label{accessible-if-preserved-by-kappa-directed}
Let $\ck$ be a locally presentable abelian category and $(\cf,\cc)$ 
a cotorsion theory generated by a set such that $\cf$ is closed
under $\kappa$-directed colimits for some regular cardinal~$\kappa$.
Suppose that any object of $\ck$ is a quotient of an object from~$\cf$. 
Then $\cf$ is accessible.
\end{coro}

\begin{proof}
According to the proof of Proposition~\ref{prop3.4}, one can choose
a generating set $\cs$ in such a way that any object of $\ck$ is
a quotient of a coproduct of objects from~$\cs$.
Then, by the proof of Theorem~\ref{precise-et-loc-pres}, an object
$F$ in $\ck$ is $\ci$-cofibrant, i.~e., the morphism $0\to F$
belongs to $\cof(\ci)$, if and only if $F$ is from~$\cf$.
The desired assertion now follows from~\cite[Corollary~5.1]{MRV}.
\end{proof}

Assuming Vop\v enka's principle, any full subcategory of $\ck$
closed under $\kappa$-directed colimits is accessible, and so is
any full subcategory of $\ck$ closed under $\kappa$-pure subobjects
(see~\cite[Theorem~6.17]{AR}).

\begin{coro} \label{accessible-if-contains-pure-subobjects}
Assume that there is a proper class of almost strongly compact cardinals.
Let $\ck$ be a locally presentable abelian category and $(\cf,\cc)$
a cotorsion theory generated by a set such that $\cf$ is closed under
$\kappa$-pure subobjects in $\ck$ for some regular cardinal~$\kappa$.
Suppose that any object of $\ck$ is a quotient of an object from~$\cf$.
Then the category $\cf$ is accessible and accessibly embedded into~$\ck$.
\end{coro}

\begin{proof}
Let $\lambda\geq\kappa$ be a regular cardinal such that $\ck$ is
locally $\lambda$-presentable and there is a set of morphisms $\ci$
with $\lambda$-presentable domains and codomains in $\ck$ such that
$\cf$ is the class of $\ci$-cofibrant objects.
Following~\cite[Proposition~3.4]{R2}, there is a $\lambda$-accessible
functor $F:\ck\to\ck$ such that the class $\cf$ consists precisely of
the retracts of objects in the image of~$F$.
Since $\cf$ is closed under $\kappa$-pure subobjects, we can conclude
that $\cf$ is the $\lambda$-pure powerful image of~$F$.
According to~\cite[Proposition~2.4, Theorem~3.2, and Remark~3.3]{BTR},
$\cf$ is accessible and accessibly embedded.
\end{proof}

\begin{rem}
 Following the proof of Theorem~\ref{precise-et-loc-pres} more closely,
in both Corollaries~\ref{accessible-if-preserved-by-kappa-directed}
and~\ref{accessible-if-contains-pure-subobjects} one can drop
the assumption that every object of $\ck$ is a quotient of an object
of $\cf$, replacing in their assertions the left part $\cf$ of
a cotorsion theory $(\cf,\cc)$ with the class ${}^\oplus\filt(\cs)$
of all direct summands of objects from $\filt(\cs)$, where $\cs$ is
any set of objects in~$\ck$.
 So, in the case of 
Corollary~\ref{accessible-if-preserved-by-kappa-directed}, it is then 
claimed that the full subcategory ${}^\oplus\filt(\cs)$ is accessible
whenever it is closed under $\kappa$-directed colimits in~$\ck$.
 In the case of Corollary~\ref{accessible-if-contains-pure-subobjects}
the assertion is that, assuming a proper class of almost strongly
compact cardinals, ${}^\oplus\filt(\cs)$ is accessible and accessibly
embedded whenever it is closed under $\kappa$-pure subobjects.
\end{rem}

The next lemma is our version of~\cite[Lemma~1 and Proposition~2]{BBE}.

\begin{lemma} \label{deconstructible}
Let $\ck$ be a locally presentable abelian category and $\cf$ a class
of objects in $\ck$ closed under transfinitely iterated extensions,
$\kappa$-pure subobjects, and $\kappa$-pure quotients for some regular
cardinal~$\kappa$.
Assume that the class of all $\cf$-monomorphisms $X\to F$ is closed
under directed colimits in the comma-category $\ck\downarrow F$ for
any $F\in\cf$.  Then the class $\cf$ is deconstructible.
\end{lemma}

\begin{proof}
 Let $\lambda\geq\kappa$ be a regular cardinal for which
$\ck$ is locally $\lambda$-presentable.
 Let $\alpha\geq\lambda$ be a regular cardinal such that any
morphism from an $\alpha$-presentable object to an object $K$ in $\ck$
factors through an $\alpha$-presentable $\lambda$-pure subobject of~$K$
(see~\cite[Theorem~2.33]{AR}).
 Then we claim that $\cf=\filt(\cs)$, where $\cs$ is a representative
set of all objects in $\cf$ that are $\alpha$-presentable in~$\ck$.
 Indeed, let an object $F\in\cf$ be $\mu$-presentable in~$\ck$,
where $\mu\trianglerighteq\lambda$ is a regular cardinal greater than
the cardinality $\gamma$ of the set of all morphisms between
$\lambda$-presentable objects in~$\ck$.
 We construct a smooth chain $(f_{ij}:F_i\to F_j)_{i<j<\mu}$ and
a compatible cocone $(f_i:F_i\to F)$ such that $F_i\in\cf$,
the morphisms $f_{i,i+1}$ are $\cs$-monomorphisms, and $f_i$ are
$\cf$-monomorphisms.

 Set $F_0=0$.
 On a successor step~$i+1$, if the morphism $f_i$ is an isomorphism,
put $F_{i+1}=F_i$, \ $f_{i+1}=f_i$, and $f_{i,i+1}=\id$.
 Otherwise, the cokernel $\coker(f_i)$ is nonzero, so there exists
a nonzero morphism from a $\lambda$-presentable object to
$\coker(f_i)$, which consequently factors through a nonzero
$\alpha$-presentable $\lambda$-pure subobject $s_i:S_i\to\coker(f_i)$.
 Since $\cf$ is closed under $\lambda$-pure subobjects, we have
$S_i\in\cs$, and since $\cf$ is closed under $\lambda$-pure quotients,
$\coker(s_i)\in\cf$.
 Let $f_{i+1}:F_{i+1}\to F$ be the pullback of the morphism $s_i$ with
respect to the epimorphism $F\to\coker(f_i)$.
 Then $f_{i+1}$ is a monomorphism with $\coker(f_{i+1})=\coker(s_i)$,
there exists a unique morphism $f_{i,i+1}:F_i\to F_{i+1}$ with
$f_{i+1}f_{i,i+1}=f_i$, and $f_{i,i+1}$ is a monomorphism with
$\coker(f_{i+1})=S_i$.
 Since $F_{i+1}$ is an extension of $S_i$ and $F_i$, and $\cf$ is closed
under extensions, we have $F_{i+1}\in\cf$.
 On a limit step, we use the assumptions that $\cf$ is closed under
transfinitely iterated extensions and the class of all
$\cf$-monomorphisms is closed under directed colimits in
$\ck\downarrow F$.

 Since $F$ cannot have an increasing chain of subobjects of length $\mu$
(see Lemma~\ref{lq1.1}), there exists an ordinal $j<\mu$ such that
$f_j:F_j\to F$ is an isomorphism.
 We have proved that $F$ is a transfinitely iterated extension of
the objects $S_i$, \ $i<j$.
\end{proof}

 The following lemma and proposition provide the most
straightforward extension of the first half of the proof of
the Eklof--Trlifaj theorem~\cite[Theorem~10]{ET} to locally presentable
abelian categories.

\begin{lemma} \label{cell-sc-dc}
 Let $\ck$ be a cocomplete abelian category and $\ch$ a class of objects
closed under extensions in~$\ck$.
 Let $K$ be an object in $\ck$ such that the class of all
$\ch$-monomorphisms $K\to X$ is closed under directed colimits in
$K\downarrow\ck$.
 Then any morphism $K\to X$ belonging to\/ $\cell(\ch\lMono)$
is an $\ch$-monomorphism.
\end{lemma}

\begin{proof}
 It is an immediate corollary of the definitions.
\end{proof}

\begin{propo} \label{et-sc-dc}
Let $\ck$ be a locally presentable abelian category and $(\cf,\cc)$
the cotorsion theory generated by a set of objects $\cs$ in~$\ck$.
Assume that the set $\cs$ is contained in a class of objects $\ch$
closed under extensions and strongly closed under directed colimits
in~$\ck$.
Then any object of $\ck$ has a special $\cc$-preenvelope.
\end{propo}

\begin{proof}
 To be more precise, let $K$ be an object in $\ck$ such that the set
$\cs$ is contained in a class of objects $\ch$ closed under extensions
in $\ck$ for which the class of all $\ch$-mono\-morphisms $K\to X$ is
closed under directed colimits in $K\downarrow\ck$.
 Then any morphism $K\to X$ belonging to $\cell(\cs\lMono)$ is
a monomorphism by Lemma~\ref{cell-sc-dc}.
 Following the proof of Theorem~\ref{precise-et-loc-pres},
the morphism $K\to0$ can be decomposed as $K\to C\to0$, where
the morphism $K\to C$ belongs to $\cell(\cs\lMono)$ and
the object $C$ belongs to the class~$\cc$.
 Furthermore, the cokernel of any morphism from $\cell(\cs\lMono)$
belongs to $\filt(\cs)$.
 We have constructed a short exact sequence $0\to K\to C\to F'\to0$
with $C\in\cc$ and $F'\in\filt(\cs)$.
 Finally, $\filt(\cs)\subseteq\cf$ by
Theorem~\ref{precise-et-loc-pres}(a).
\end{proof}

To finish this section, let us formulate the somewhat more precise
versions of Corollaries~\ref{cor3.6} and~\ref{cor3.9} that
the techniques of their proofs allow to obtain.

\begin{coro}
 Let $\ck$ be a locally presentable category and $\ca$ a class of
objects in $\ck$ closed under directed colimits.
 Then whenever an object of $\ck$ that has an $\ca$-precover, it
also has an $\ca$-cover. \qed
\end{coro}

\begin{coro} \label{precise-envelope}
 Let $\ck$ be a locally presentable abelian category and $(\cf,\cc)$
a cotorsion theory in~$\ck$.
 Let $K$ be an object of $\ck$ such that the class of all
$\cf$-monomorphisms $K\to X$ is closed under directed colimits
in the comma-category $K\downarrow\ck$.
 Assume that the object $K$ has a special $\cc$-preenvelope.
 Then $K$ also has a $\cc$-envelope. \qed
\end{coro}

\begin{rem}
 Weakening the assumptions of Corollary~\ref{precise-envelope} a bit
further, let $\cf$ be a class of objects closed under extensions
in a locally presentable abelian category~$\ck$.
 Let $K\in\ck$ be an object such that the class of all
$\cf$-monomorphisms $K\to X$ is closed under directed colimits
in $K\downarrow\ck$.
 Set $\cc=\cf^\perp$, and assume that there exists an $\cf$-monomorphism
$K\to C$ with $C\in\cc$.
 Then the object $K$ has a $\cc$-envelope.
\end{rem}

\begin{coro}
Let $\ck$ be a locally presentable abelian category and $(\cf,\cc)$
a cotorsion theory generated by a set of objects in~$\ck$.
Let $K$ be an object of $\ck$ such that the class of all
$\cf$-monomorphisms $K\to X$ is closed under directed colimits in
the comma-category $K\downarrow\ck$.
Then $K$ has a $\cc$-envelope.
\end{coro}

\begin{proof}
This follows from Corollary~\ref{precise-envelope} and (the proof of)
Proposition~\ref{et-sc-dc}.
\end{proof}

\section{Contramodules over topological rings}
\label{contramodule-basics}

Let $\fr$ be an associative, unital topological ring where open right
ideals $\fii\subseteq\fr$ form a base of neighborhoods of zero.
We will always assume that $\fr$ is complete and separated, that is
the natural map $\fr\to\lim_{\fii}\fr/\fii$ from $\fr$ to the limit of
its quotients by its open right ideals is an isomorphism (of abelian
groups, or of right $\fr$-modules).

For any topological abelian group $\fa$ where open subgroups
$\fu\subseteq\fa$ form a base of neighborhoods of zero and any set $X$,
we denote by $\fa[[X]]\subseteq\fa^X$ the set of all infinite formal
linear combinations $\sum_x a_x x$ of elements $x$ of $X$ with families
of coefficients $a_x\in\fa$ converging to zero in the topology of~$\fa$.
The latter condition means that for any open subgroup $\fu\subseteq\fa$
the set of all $x\in X$ for which $a_x\notin\fu$ must be finite.
When the topological group $\fa$ is complete and separated, that is
$\fa=\lim_{\fu}\fa/\fu$, one has $\fa[[X]]=\lim_{\fu}\,\fa/\fu[X]$,
where $A[X]=A^{(X)}$ denotes the set of all finite linear combinations
of elements of $X$ with coefficients in a group~$A$.
Given a map of sets $f:X\to Y$, there is the induced ``direct image''
map $\fa[[f]]:\fa[[X]]\to\fa[[Y]]$ taking $\sum_x a_x x$ to
$\sum_y(\sum_{f(x)=y}a_x)y$, where the sum of the coefficients
$a_x\in\fa$ over all the preimages $x\in X$ of a given element $y\in Y$
is defined as the limit of finite partial sums converging
in the topology of~$\fa$.

For any associative, unital ring $R$ and any set $X$ there is a natural
map of ``opening of parentheses'' $\phi_R(X):R[R[X]]\to R[X]$ defining,
together with the ``point measure'' map $\epsilon_R(X):X\to R[X]$,
the structure of a monad on the category of sets on the functor
$X\mapsto R[X]$.
The category of left $R$-modules is equivalent to the category of
algebras/modules over this monad on the category of sets.

Similarly, for any topological ring $\fr$ with the topology of
the kind described above and for any set $X$ there is an ``opening of
parentheses'' map $\phi_{\fr}(X):\fr[[\fr[[X]]]]\to\fr[[X]]$ involving
infinite summation of coefficients understood as the limit of finite
partial sums in~$\fr$ (our conditions on the topology of $\fr$ are
designed to ensure the convergence).
More specifically, in order to obtain the map $\phi_{\fr}(X)$ it suffices
to construct a morphism of diagrams $\fr/\fii[\fr[[X]]]\to\fr/\fii[X]$
indexed by the partially ordered set of all open right ideals
$\fii\subseteq\fr$.
This is easily done using the fact that for any element $r\in\fr$ and
any open right ideal $\fii$ in $\fr$ there exists an open right ideal
$\fii'$ in $\fr$ such that $r\fii'\subseteq\fii$.
The natural transformation $\phi_{\fr}(X)$, together with the ``point
measure'' map $\epsilon_{\fr}(X):X\to\fr[[X]]$, define the structure of
a monad on the category of sets on the functor $X\mapsto\fr[[X]]$.

\begin{defi}
 The category of \textit{left\/ $\fr$-contramodules} $\fr\lContra$
is defined as the category of algebras/modules over the monad
$\fr[[{-}]]:\Set\to\Set$.
 In other words, a left $\fr$-contramodule $\fp$ is a set endowed with
a \textit{left contraaction} map $\pi_{\fp}:\fr[[\fp]]\to\fp$ satisfying
the equations of \textit{contraassociativity}
$\pi_{\fp}\circ\phi_{\fr}(\fp)=\pi_{\fp}\circ\fr[[\pi_{\fp}]]$
$$
 \fr[[\fr[[\fp]]]]\rightrightarrows\fr[[\fp]]\longrightarrow\fp
$$
and
\textit{contraunitality} $\pi_{\fp}\circ\epsilon_{\fr}(\fp)=\id_{\fp}$
$$
 \fp\longrightarrow\fr[[\fp]]\longrightarrow\fp.
$$
\end{defi}

One can view an $\fr$-contramodule $\fp$ as a set endowed with
\textit{infinite summation operations} assigning to any set $X$,
an $X$-indexed family of elements $r_x\in\fr$ converging to zero
in the topology of $\fr$, and an $X$-indexed family of elements 
$p_x\in\fp$ an element $\pi_{\fp}(\sum_x r_x p_x)\in\fp$.
Here $\sum_x r_x p_x\in\fr[[\fp]]$ is a notation for the image of
the element $\sum_x r_x x\in\fr[[X]]$ with respect to the ``direct
image'' map $\fr[[X]]\to \fr[[\fp]]$ induced by the map
$X\to\fp$ taking $x$ to~$p_x$.
Composing the contraaction map $\pi_\fp$ with the inclusion
$\fr[\fp]\to\fr[[\fp]]$ (i.~e., restricting oneself to summation
operations with finite coefficient families), one constructs
the underlying left $\fr$-module structure on
a left $\fr$-contramodule~$\fp$.
This provides the forgetful functor $\fr\lContra\to\fr\lMod$.

The \textit{free\/ $\fr$-contramodule} generated by a set $X$ is
the free algebra/module over the monad $\fr[[{-}]]$ generated by $X$,
that is the left $\fr$-contramodule $\fr[[X]]$ with the contraaction
map $\pi_{\fr[[X]]}=\phi_{\fr}(X)$.
The free $\fr$-contramodule $\fr[[X]]$ is the coproduct of $X$ copies
of the free $\fr$-contramodule with one generator $\fr=\fr[[\{*\}]]$
in the category of left $\fr$-contramodules.
The forgetful functor $\fr\lContra\to\fr\lMod$ can be also
constructed as the corepresentable functor $\fr\lContra(\fr,-)$.

Let $\lambda$ be the successor cardinal of the cardinality of a base
of neighborhoods of zero in~$\fr$.
Then the category of left $\fr$-contramodules $\fr\lContra$ is
a locally $\lambda$-presentable abelian category (because
the cardinality of the support $\{x\mid r_x\ne0\}$ of any family of
coefficients $r_x\in\fr$ converging to zero in the topology of $\fr$
is smaller than~$\lambda$).
The free $\fr$-contramodule with one generator $\fr$ is
a $\lambda$-presentable projective generator of $\fr\lContra$.
A left $\fr$-contramodule is projective if and only if it is
a direct summand of the free $\fr$-contramodule $\fr[[X]]$
for some set~$X$.
The forgetful functor $\fr\lContra\to\fr\lMod$ is exact and
preserves arbitrary products (i.~e., preserves finite colimits and
all limits).
It also preserves $\lambda$-filtered colimits.

\begin{exam} \label{noetherian-adic-contra}
 Let $R$ be a Noetherian commutative ring and $I\subset R$ an ideal.
 Consider the $I$-adic completion $\fr=\lim_m R/I^m$ of the ring $R$,
viewed as a topological ring in the $I$-adic topology (or, which is
the same, the topology of limit of discrete rings $R/I^m$).
 Then the forgetful functor $\fr\lContra\to R\lMod$ is fully
faithful.
 Its image is the full subcategory $\cp\subseteq R\lMod$ of
$I$-contramodule $R$-modules described in
Example~\ref{ext-0-1-orthogonal-category}(3)
(see~\cite[Theorem~B.1.1]{P2}).
\end{exam}

 Classical counterexamples~\cite[Section~1.5]{P4} show that
the contramodule infinite summation cannot be understood as any kind
of limit of finite partial sums, as the image of an infinite formal
linear combination $\sum_x r_x p_x\in\fr[[\fp]]$ under the contraaction
map $\pi_{\fp}:\fr[[\fp]]\to\fp$ may well be nonzero while all
the products $r_x\cdot p_x\in\fp$, taken in the underlying $\fr$-module
structure of the $\fr$-contramodule $\fp$, vanish.
 In fact, the $I$-adic topology on an $I$-contramodule $R$-module
$Q\in\cp\subseteq R\lMod$ does not need to be separated, i.~e.,
the intersection $\cap_m I^mQ\subseteq Q$ may be nonzero (even though
$IQ=Q$ implies $Q=0$ \cite[part~(b) of Lemma in Section~2.1]{P4}).

 Let, as above, $\fr$ be an associative and unital, complete and
separated topological ring where open right ideals form a base of
neighborhoods of zero.

\begin{defi}
 A right $\fr$-module $N$ is called \textit{discrete} if
the multiplication map $N\times\fr\to N$ is continuous
in the given topology of $\fr$ and the discrete topology of~$N$.
 Equivalently, this means that for any element $b\in N$ there exists
an open right ideal $\fii\subseteq\fr$ for which $b\fii=0$ in~$N$.
 We denote the full subcategory of discrete right $\fr$-modules in
the category of right $\fr$-modules by
$\Discrd\fr\subseteq\Modr\fr$.
\end{defi}

 The full subcategory of discrete right $\fr$-modules $\Discrd\fr$
is closed under all colimits, finite limits, and extensions in
$\Modr\fr$.
 The category $\Discrd\fr$ is a Grothendieck abelian category and
a coreflective subcategory in $\Modr\fr$.

 For any abelian group $A$ and discrete right $\fr$-module $N$,
there is a natural structure of left $\fr$-contramodule on
the set of all abelian group homomorphisms $\Ab(N,A)$.
 Given a set $X$, an $X$-indexed family of coefficients $r_x\in\fr$
converging to zero in the topology of $\fr$, and an $X$-indexed
family of homomorphisms $f_x\in\Ab(N,A)$, one defines
the homomorphism $f=\pi_{\Ab(N,A)}(\sum_x r_xf_x)$ by the rule
$f(b)=\sum_xf_x(br_x)\in A$ for all $b\in N$.
 Since the right $\fr$-module $N$ is discrete, one has $br_x=0$
for all but a finite subset of indices $x\in X$, so the sum is
actually finite.

\begin{defi}
 The \textit{contratensor product} $N\odot_{\fr}\fp$ of a discrete
right $\fr$-module $N$ and a left $\fr$-contramodule $\fp$ is
an abelian group constructed as the quotient group of the tensor product
group over the ring of integers $N\otimes_{\integers}\fp$ by the subgroup
generated by all elements of the form
$b\otimes\pi_{\fp}(\sum_x r_xp_x)-\sum_x br_x\otimes p_x$, where $X$ is
a set, $r_x\in\fr$ is an $X$-indexed family of elements converging to
zero in $\fr$, and $p_x$ is an $X$-indexed family of elements in~$\fp$.
 One has $br_x=0$ for all but a finite subset of indices~$x$, so
the expression in the right-hand side is well-defined.
\end{defi}

For any discrete right $\fr$-module $N$, left $\fr$-contramodule $\fp$,
and abelian group $A$, there is a natural isomorphism of abelian groups
$$
 \Ab(N\odot_{\fr}\fp,\,A)\simeq\fr\lContra(\fp,\Ab(N,A)).
$$
In other words, the functor of contratensor product
$N\odot_{\fr}{-}:\fr\lContra\to\Ab$ is left adjoint to
the functor of homomorphisms $\Ab(N,-):\Ab\to\fr\lContra$.

Given an open right ideal $\fii\subseteq\fr$ and a left
$\fr$-contramodule $\fp$, we denote by $\fii\tim\fp$ the image of
the subgroup $\fii[[\fp]]\subseteq\fr[[\fp]]$ under the contraaction
map $\fr[[\fp]]\to\fp$.  As usually, given a left $\fr$-module $M$,
the notation $\fii M$ stands for the subgroup of $M$ generated by
elements of the form $rm$, where $r\in\fii$ and $m\in M$.
So for any left $\fr$-contramodule $\fp$ one has $\fii\fp\subseteq
\fii\tim\fp\subseteq\fp$.
The quotient $\fr/\fii$ is a discrete right $\fr$-module, and there
are natural isomorphisms of abelian groups
$(\fr/\fii)\otimes_\fr M = M/\fii M$ and
$(\fr/\fii)\odot_\fr\fp = \fp/(\fii\tim\fp)$.

\begin{propo} \label{lambda-pure-contratensor}
 Let $\lambda$ be the successor cardinal of the cardinality of a base
of neighborhoods of zero in\/~$\fr$.
 Let\/ $f:\fq\to\fp$ be a $\lambda$-pure monomorphism in the category
of left\/ $\fr$-contramodules.
 Then for any discrete right\/ $\fr$-module $N$ the induced morphism
of the contratensor products
$$
 N\odot_{\fr}f: N\odot_{\fr}\fq\longrightarrow N\odot_{\fr}\fp
$$
is a monomorphism of abelian groups.
\end{propo}

\begin{proof}
 Let $a_1\otimes q_1+\dotsb+a_n\otimes q_n$ be an element of
$N\otimes_{\integers}\fq$ whose image in $N\odot_{\fr}\fq$ is
annihilated by the morphism $N\odot_{\fr}f$.
 Then there exist elements $b_1$,~\dots, $b_m\in N$, index sets
$X_1$,~\dots, $X_m$, families of elements $r_{k,x}\in\fr$, \ $x\in X_k$
converging to zero in the topology of $\fr$, and families of
elements $p_{k,x}\in\fp$, \ $x\in X_k$ such that
$$
 \sum_{j=1}^n a_j\otimes f(q_j) =
 \sum_{k=1}^m b_k\otimes\pi_{\fp}\Bigg(\sum_{x\in X_k}r_{k,x}p_{k,x}\Bigg)
 - \sum_{k=1}^m\sum_{x\in X_k} b_kr_{k,x}\otimes p_{k,x}
$$
in $N\otimes_{\integers}\fp$.
 Replacing the set $X_k$ with its subset consisting of all the indices
$x$ for which $r_{k,x}\ne0$, we may assume without loss of generality
that the cardinality of $X_k$ is smaller than~$\lambda$.
 Set $p_k=\pi_{\fp}(\sum_{x\in X_k}r_{k,x}p_{k,x})\in\fp$, denote by
$Z_k$ the (finite) set of all $z\in X_k$ for which $b_kr_{k,z}\ne0$,
and put $b_{k,z}=b_kr_{k,z}$ for every $z\in Z_k$.
 Then the following relation depending only on the additive group
structures of $N$ and $\fp$
$$
 \sum_{j=1}^n a_j\otimes f(q_j) =
 \sum_{k=1}^m b_k\otimes p_k
 - \sum_{k=1}^m\sum_{z\in Z_k} b_{k,z}\otimes p_{k,z}
$$
holds in $N\otimes_{\integers}\fp$.
 Hence there exist a free abelian group $T$ with a finite set of
generators $t_1$,~\dots, $t_l$, elements $q_j'$, $p_k'$, $p_{k,z}'$ in $T$,
where $z\in Z_k$, a subgroup $S\subset T$ generated by some elements
$s_1$,~\dots, $s_e$, and an abelian group homomorphism $t':T\to\fp$
such that $t'(q_j')=f(q_j)$, \ $t'(p_k')=p_k$, \ $t'(p_{k,z}')=p_{k,z}$,
\ $t'(S)=0$, and the relation
$$
 \sum_{j=1}^n a_j\otimes q_j' =
 \sum_{k=1}^m b_k\otimes p_k'
 - \sum_{k=1}^m\sum_{z\in Z_k} b_{k,z}\otimes p_{k,z}'
$$
holds in $N\otimes_{\integers}(T/S)$.

 Now let $\ft$ denote the free left $\fr$-contramodule generated by
the symbols $t_1''$,~\dots, $t_l''$ and $p''_{k,y}$, \
$y\in X_k\setminus Z_k$.
 Denote by $t'':T\to\ft$ the abelian group homomorphism taking
$t_i$ to $t_i''$, \ $1\le i\le l$, and set $q''_j=t''(q'_j)$, \
$p''_k=t''(p'_k)$, \ $s_i''=t''(s_i)$, and
$p''_{k,z}=t''(p'_{k,z})\in\ft$ for $z\in Z_k$.
 Let $\fs$ be the free left $\fr$-contramodule generated by
the symbols $s_1'$,~\dots, $s_e'$ and $w_1'$,~\dots, $w_m'$.
 Denote by $s:\fs\to\ft$ the $\fr$-contramodule morphism taking
$s_i'$ to $s_i''$, \ $1\le i\le e$, and $w_k'$ to $p_k''-
\pi_{\ft}(\sum_{x\in X_k}r_{k,x}p''_{k,x})$, \ $1\le k\le m$.
 Let $t:\ft\to\fp$ be the $\fr$-contramodule morphism taking
$t_i''$ to $t'(t_i)$, \ $1\le i\le l$, and $p''_{k,y}$ to $p_{k,y}$,
\ $y\in X_k\setminus Z_k$.
 Then one has $t(q_j'')=f(q_j)$, \ $t(p_k'')=p_k$, \ $t(s_i'')=
t'(s_i)=0$, and $t(p_{k,z}'')=p_{k,z}$, \ $z\in Z_k$.
 Hence the composition $ts:\fs\to\fp$ vanishes, so we have the induced
$\fr$-contramodule morphism $v:\fv\to\fp$, where $\fv=\ft/s(\fs)$.

 By construction, the relations
\begin{multline*}
 \sum_{j=1}^n a_j\otimes q_j'' =
 \sum_{k=1}^m b_k\otimes p_k''
 - \sum_{k=1}^m\sum_{z\in Z_k} b_{k,z}\otimes p_{k,z}'' \\ =
 \sum_{k=1}^m b_k\otimes\pi_{\fp}\Bigg(\sum_{x\in X_k}r_{k,x}p_{k,x}''\Bigg)
 - \sum_{k=1}^m\sum_{x\in X_k} b_kr_{k,x}\otimes p''_{k,x}
\end{multline*}
hold in $N\otimes_{\integers}\fv$.
 Hence the image of the element $a_1\otimes q_1''+\dotsb+
a_n\otimes q_n''\in N\otimes_{\integers}\fv$ vanishes in
$N\odot_{\fr}\fv$.

 Let $\fu$ denote the free left $\fr$-contramodule generated by
the symbols $u_j$, \ $1\le j\le n$.
 Denote by $u:\fu\to\fq$ the $\fr$-contramodule morphism taking $u_j$
to $q_j$, by $q:\fu\to\ft$ the $\fr$-contramodule morphism taking $u_j$
to $q''_j$, and by $g:\fu\to\fv$ the composition $\fu\to\ft\to\fv$.
 We have constructed a commutative diagram
$$
\xymatrix{
\fq\ar [r]^{f}& \fp\\
\fu\ar [u]^{u}\ar [r]_{g}& \fv\ar [u]_{v}
}
$$
in the category of left $\fr$-contramodules.
 The objects $\fu$, $\fv\in\fr\lContra$ are $\lambda$-presentable
as cokernels of morphisms between coproducts of less than $\lambda$
copies of the $\lambda$-presentable projective generator~$\fr$.
 By assumption, the morphism $u$ factorizes through $g$, and it follows
that the image of the element $a_1\otimes q_1+\dotsb+a_n\otimes q_n
\in N\otimes_{\integers}\fq$ vanishes in $N\odot_{\fr}\fq$.
\end{proof}

A left $\fr$-contramodule $\ff$ is called (\textit{contra})\textit{flat}
if the functor of contratensor product $N\mapsto N\odot_{\fr}\ff$ is
exact on the abelian category of discrete right $\fr$-modules~$N$.

\begin{lemma} \label{flats-colimit-closed}
 The class of flat\/ $\fr$-contramodules is closed under filtered
colimits in\/ $\fr\lContra$.
\end{lemma}

\begin{proof}
 Let $(\ff_i)_{i\in\cc}$ be a filtered diagram of flat left
$\fr$-contramodules.
 For any discrete right $\fr$-module $N$, the functor of contratensor
product $N\odot_{\fr}{-}:\fr\lContra\to\Ab$ is a left adjoint, so
it preserves colimits.
 Hence we have $N\odot_{\fr}\colim_i\ff_i=\colim_i(N\odot_{\fr}\ff_i)$.
 Now for any short exact sequence $0\to L\to M\to N\to0$ in
$\Discrd\fr$ and every $i\in\cc$ the short sequence of contratensor
products $0\to L\odot_{\fr}\ff_i\to M\odot_{\fr}\ff_i\to N\odot_{\fr}\ff_i
\to0$ is exact in $\Ab$, since the $\fr$-contramodule $\ff_i$ is flat.
 It remains to notice that the colimit of a filtered diagram of
short exact sequences of abelian groups is a short exact sequence
in order to conclude that $\colim_i\ff_i$ is a flat $\fr$-contramodule.
\end{proof}

\begin{exams} \label{contratensor-examples}
(1) Let $\fr$ be an arbitrary associative ring endowed with the discrete
topology.
 Then any right $\fr$-module is discrete, so $\Discrd\fr=
\Modr\fr$, and a left $\fr$-contramodule is the same thing as
a left $\fr$-module, so $\fr\lContra=\fr\lMod$.
 The contratensor product functor $\odot_{\fr}$ is isomorphic to
the tensor product functor $\otimes_{\fr}$, hence the flat left
$\fr$-contramodules are simply the flat $\fr$-modules.

\medskip

(2) Let $\fr$ be a complete and separated topological associative ring
where open two-sided ideals form a base of neighborhoods of zero.
For any open two-sided ideal $\fii\subseteq\fr$ and left
$\fr$-contramodule $\fp$, the quotient group $\fp/(\fii\tim\fp)$
has a natural left $\fr/\fii$-module structure.
The reduction functor $\fp\mapsto\fp/(\fii\tim\fp)$ is left adjoint
to the fully faithful inclusion $\fr/\fii\lMod\to\fr\lContra$.
For any discrete right $\fr$-module $N$, the subgroup $N_{\fii}
\subseteq N$ of all elements annihilated by $\fii$ in $N$ is
a right $\fr/\fii$-module.
The functor $N\mapsto N_{\fii}$ is right adjoint to the fully faithful
inclusion $\Modr\fr/\fii\to\Discrd\fr$.

 The contratensor product $N\odot_{\fr}\fp$ can be computed as
the filtered colimit of tensor products over the discrete quotient rings
$$
 N\odot_{\fr}\fp=\colim_{\fii}N_{\fii}\otimes_{\fr/\fii}(\fp/\fii\tim\fp).
$$
A left $\fr$-contramodule $\ff$ is flat if and only if its reduction
$\ff/(\fii\tim\ff)$ is a flat $\fr/\fii$-module for every open
two-sided ideal $\fii\subseteq\fr$.

\medskip

(3) Let $\fr=\lim_m R/I^m$ be the $I$-adic completion of a Noetherian
commutative ring $R$, as in Example~\ref{noetherian-adic-contra}.
 Then $\Discrd\fr$ is the category of \textit{$I$-torsion
$R$-modules}, i.~e., the full subcategory in $\Modr R$ formed by
all the $R$-modules $N$ such that for any $b\in N$ and $r\in I$
there exists an integer $m\ge1$ for which $br^m=0$.
 Furthermore, one has $\fj\tim\fp=\fj\fp$ for any $\fr$-contramodule
$\fp$ and open ideal $\fj\subseteq\fr$, so
$N\odot_{\fr}\fp=N\otimes_{\fr}\fp$ for any discrete $\fr$-module $N$
and $\fr$-contramodule~$\fp$.
 According to~\cite[Lemma~B.9.2]{P2}, an $\fr$-contramodule is flat
if and only if its underlying $R$-module is flat.
 (See~\cite[Section~C.5]{P3} for a generalization to adic completions
of noncommutative Noetherian rings by centrally generated ideals.)

 An $\fr$-contramodule $\fp$ is projective if and only if $\fp/I^m\fp$
is a projective $R/I^m$-module for every $m\ge1$ \
\cite[Corollary~B.8.2]{P2}.
 In particular, when $I$ is a maximal ideal in $R$,
an $\fr$-contramodule is flat if and only if it is projective
(and if and only if it is free~\cite[Lemma~1.3.2]{P2}).

\medskip

(4) When $R=k[x_1,\dotsc,x_n]$ is the ring of polynomials in several
variables over a field $k$ and $I$ is the ideal generated by
$x_1$,~\dots, $x_n$ in $R$, the category of $I$-torsion $R$-modules
$\Discrd\fr$ is equivalent to the category of comodules over
the coalgebra $C$ over $k$ whose dual topological algebra $C^*=\fr$
is the algebra of formal power series $\fr=k[[x_1,\dotsc,x_n]]$.
 The category of $I$-contramodule $R$-modules $\fr\lContra$ is
equivalent to the category of $C$-contramodules~\cite[Sections~10.1
and~A.1]{P}.

 More generally, for any coassociative coalgebra $C$ over $k$,
a right $C$-comodule is the same thing as a discrete right 
$C^*$-module, and a left $C$-contramodule is the same thing as
a left $C^*$-contramodule \cite[Sections~1.3, 1.4, and~2.3]{P4}.
 A $C^*$-contramodule is (contra)flat if and only if it is
projective~\cite[Lemma~A.3]{P}.
\end{exams}

\section{Flat Contramodules}
\label{flat-contramodules-secn}

Let $\kappa$ be a regular cardinal and $\cn$ a $\kappa$-Grothendieck
abelian category.
We fix the cardinal $\kappa$ and denote by $\Fun(\cn)\subseteq\Ab^\cn$
the category of all additive covariant functors $F:\cn\to\Ab$
preserving $\kappa$-filtered colimits.
Since such functors are determined by their restrictions to
the full subcategory of $\lambda$-presentable objects in $\cn$ for
a large enough cardinal $\lambda$, the category $\Fun(\cn)$
has small hom-sets (i.~e., is well-defined as a category in
the conventional sense of the word).

The category $\Fun(\cn)$ is abelian and cocomplete.
A short sequence of functors $0\to E\to F\to G\to 0$ is exact in
$\Fun(\cn)$ if and only if the short sequence of abelian groups
$0\to E(N)\to F(N)\to G(N)\to 0$ is exact in $\Ab$ for
every object $N$ in~$\cn$.
The functors of filtered colimit are exact in $\Fun(\cn)$.

Let $\Rex(\cn)\subseteq\Fun(\cn)$ be the full subcategory of functors
preserving also finite colimits, and $\Ex(\cn)\subseteq\Rex(\cn)$
the full subcategory of functors preserving both the finite colimits
and finite limits.
The subcategory $\Rex(\cn)$ is closed under all colimits in $\Fun(\cn)$,
while the subcategory $\Ex(\cn)$ is closed under filtered colimits.

\begin{lemma} \label{exact-seq-functors}
 Let\/ $0\to H\to G\to F\to 0$ be a short exact sequence in\/
$\Fun(\cn)$.  Then \par
\textup{(a)} $H\in\Rex(\cn)$ whenever $G\in\Rex(\cn)$
and $F\in\Ex(\cn)$; \par
\textup{(b)} $H\in\Ex(\cn)$ whenever $G$, $F\in\Ex(\cn)$; \par
\textup{(c)} $G\in\Rex(\cn)$ whenever $F$, $H\in\Rex(\cn)$; \par
\textup{(d)} $G\in\Ex(\cn)$ whenever $F$, $H\in\Ex(\cn)$.
\end{lemma}

\begin{proof}
 All the assertions hold due to related properties of short exact
sequences of short sequences of abelian groups.
 In particular, part~(a) holds because the kernel of a surjective
morphism from a short right exact sequence to a short exact sequence
of abelian groups is right exact; part~(b) holds because the kernel of
a surjective morphism of short exact sequences is exact; part~(c)
holds because an extension of short right exact sequences is
right exact, etc.
 All of these can be easily established by considering the related
$9$-term long exact sequence of cohomology groups.
\end{proof}

 Let $\fr$ be a complete, separated topological ring with a countable
base of neighborhoods of zero formed by open right ideals; so
$\fr\lContra$ is a locally $\aleph_1$-presentable abelian category
and $\Discrd\fr$ is a Grothendieck abelian category.
 Take $\cn=\Discrd\fr$ and set $\kappa=\omega$, so
$\Fun(\Discrd\fr)$ denotes the category of additive covariant
functors $\Discrd\fr\to\Ab$ preserving filtered colimits and
$\Rex(\Discrd\fr)$ is the category of functors 
$\Discrd\fr\to\Ab$ preserving all colimits.
 The results below in this section generalize those
of~\cite[Section~D.1]{P3}, where the case of a topological ring with
a countable base of neighborhoods of zero formed by open two-sided
ideals was considered.

 Given a left $\fr$-contramodule $\fp$, we will denote by
$\CT(\fp):\Discrd\fr\to\Ab$ the functor of contratensor product
$N\mapsto N\odot_{\fr}\fp$.
 Since the contravariant functor $N\mapsto\Ab(\CT(\fp)(N),A)=
\fr\lContra(\fp,\Ab(N,A))$ takes colimits in $\Discrd\fr$ to
limits in $\Ab$ for any abelian group $A$, the functor $\CT(\fp)$
preserves colimits.
 We have constructed an additive covariant functor
$\CT:\fr\lContra\to\Rex(\Discrd\fr)$.

 Conversely, let $D$ denote the partially ordered set of open right
ideals in $\fr$ with respect to the inclusion order.
 For any open right ideal $\fii\in D$, the quotient module
$Q(\fii)=\fr/\fii$ is a discrete right $\fr$-module.
 For any pair of embedded open right ideals
$\fii'\subseteq\fii\subseteq\fr$, there is a natural morphism
$\fr/\fii'\to\fr/\fii$ in $\Discrd\fr$ forming a commutative
triangle with the projection morphisms $\fr\to\fr/\fii'$ and
$\fr\to\fr/\fii$ in $\Modr\fr$.
 So we have a naturally defined diagram $Q:D\to\Discrd\fr$.
 Given any covariant functor $F:\Discrd\fr\to\Ab$, we set
$\PL(F)=\lim_D(F\circ Q)$.

 We will say that a left $\fr$-contramodule $\fp$ is \textit{separated}
if the intersection of its subgroups $\fii\tim\fp$, taken over all
the open right ideals $\fii$ in $\fr$, vanishes.
 Clearly, any subcontramodule of a separated left $\fr$-contramodule is
separated.
 Denote the full subcategory of separated left $\fr$-contramodules
by $\fr\lSepar\subseteq\fr\lContra$.

\begin{lemma} \label{pl-ct-adjunction}
\textup{(a)} For any additive covariant functor $F:\Discrd\fr\to\Ab$,
there is a natural structure of left\/ $\fr$-contramodule on the set\/
$\PL(F)$. \par
\textup{(b)} For any additive functor $F$,
the\/ $\fr$-contramodule\/ $\PL(F)$ is separated. \par
\textup{(c)} The functor\/ $\PL:\Fun(\Discrd\fr)\to\fr\lContra$
is right adjoint to the functor\/ $\CT:\fr\lContra\to
\Fun(\Discrd\fr)$. \par
\end{lemma}

\begin{proof}
 Given an additive functor $F:\Discrd\fr\to\Ab$ and
an open right ideal $\fii\subseteq\fr$, denote by $\psi_{\fii}$
the projection map $\PL(F)\to F(\fr/\fii)$.
 Given an element $r\in\fr$, let $\fj_r\subseteq\fr$ be an open
right ideal for which $r\fj_r\subseteq\fii$; then the element $r$
acts by a natural morphism $r':\fr/\fj_r\to\fr/\fii$ in
the category $\Discrd\fr$.

 Part~(a): for any set $X$, an $X$-indexed family of elements
$r_x\in\fr$ converging to zero in the topology of $\fr$, and
an $X$-indexed family of elements $p_x\in\PL(F)$, one defines
the element $\sum_xr_xp_x\in\PL(F)$ by the rule
$$
 \psi_{\fii}\Big(\sum\nolimits_x r_xp_x\Big) =
 \sum\nolimits_x F(r_x')(\psi_{\fj_x}(p_x)),
$$
where $\psi_{\fj_x}(p_x)\in F(\fr/\fj_{r_x})$ and
$F(r_x'):F(\fr/\fj_{r_x})\to F(\fr/\fii)$.
 The sum in the right-hand side is finite, because the morphism
$r_x'$ vanishes whenever $r_x\in\fii$.
 Part~(b): one has $\psi_{\fii}(p)=0$ for any $p\in\fii\tim\PL(F)$,
so $p\in\bigcap_{\fii}(\fii\tim\PL(F))$ implies $p=0$.

 Part~(c): let $\fp$ be a left $\fr$-contramodule.
 Given a morphism $f:\CT(\fp)\to F$ in the category of functors
$\Fun(\Discrd\fr)$, consider the morphism of abelian groups
$f(\fr/\fii):\fp/(\fii\tim\fp)=\fr/\fii\odot_{\fr}\fp\to
F(\fr/\fii)$.
 The compositions $\fp\to\fp/(\fii\tim\fp)\to F(\fr/\fii)$ form
a compatible cone, defining a morphism to the projective limit
$g:\fp\to\PL(F)$.
 Conversely, let $g:\fp\to\PL(F)$ be an $\fr$-contramodule morphism.
 Let $N$ be a discrete right $\fr$-module, $b$ an element in~$N$,
and $\fii$ an open right ideal annihilating~$b$; then we have
a related morphism $b':\fr/\fii\to N$ in the category
$\Discrd\fr$.
 Composing the morphism of abelian groups $\psi_{\fii}g:\fp\to
F(\fr/\fii)$ with the morphism $F(b'):F(\fr/\fii)\to F(N)$, we
obtain a morphism of abelian groups $f'(b):\fp\to F(N)$.
 The desired morphism $f(N):N\odot_{\fr}\fp\to F(N)$ takes
an element $b\odot p\in N\odot_{\fr}\fp$ to the element
$f'(b)(p)\in F(N)$ for any $p\in\fp$.
\end{proof}

\begin{lemma}  \label{adjunction-isomorphisms}
\textup{(a)} For any functor $F\in\Rex(\Discrd\fr)$, the adjunction
morphism\/ $\CT(\PL(F))\to F$ is an isomorphism. \par
\textup{(b)} For any left\/ $\fr$-contramodule\/ $\fp$, the adjunction
morphism\/ $\fp\to\PL(\CT(\fp))$ is surjective. \par
\textup{(c)} A left\/ $\fr$-contramodule\/ $\fp$ is separated if and
only if the adjunction morphism\/ $\fp\to\PL(\CT(\fp))$ is
an isomorphism.
\end{lemma}

\begin{proof}
 Part~(a): as a colimit-preserving additive functor $\Discrd\fr
\to\Ab$ is determined by its restriction to the full subcategory
formed by the discrete right $\fr$-modules $\fr/\fii$, it suffices
to check that the natural morphism $\PL(F)/(\fii\tim\PL(F))\to
F(\fr/\fii)$ is an isomorphism.
 Since the diagram $Q$ has a countable cofinal subdiagram and
the map $F(\fr/\fii')\to F(\fr/\fii)$ is surjective for any
open right ideals $\fii'\subseteq\fii$ in $\fr$, the projection
map $\psi_{\fii}:\PL(F)\to F(\fr/\fii)$ is surjective.

 To compute its kernel, choose a chain of open right ideals
$\fii=\fii_0\supseteq\fii_1\supseteq\fii_2\supseteq\dotsb$ indexed
by the nonnegative integers and cofinal in~$D$.
 Let $p_n\in PL(F)$, $n\ge0$ be an element for which
$\psi_{\fii_n}(p_n)=0$.
 We have an exact sequence
$$
 \bigoplus\nolimits_{r\in\fii_n} \fr/\fj_r \xrightarrow{\ \ r' \ }
 \fr/\fii_{n+1} \xrightarrow{\ \ \ \ } \fr/\fii_n
 \xrightarrow{\ \ \ \ } 0
$$
in the category $\Discrd\fr$, where $\fj_r$ is an open right
ideal in $\fr$ for which $r\fj_r\subseteq\fii_{n+1}$.
 Applying the functor $F$, we obtain an exact sequence of
abelian groups
$$
 \bigoplus\nolimits_{r\in\fii_n} F(\fr/\fj_r) \xrightarrow{\ F(r') \ }
 F(\fr/\fii_{n+1}) \xrightarrow{\ \ \ \ } F(\fr/\fii_n)
 \xrightarrow{\ \ \ \ } 0.
$$
 Hence there exist finite families of elements $r_{n,1}$,~\dots,
$r_{n,k_n}\in\fii_{n+1}$ and $q'_{n,i}\in F(\fr/\fj_{r_{n,i}})$
such that $\psi_{\fii_{n+1}}(p_n)=F(r_{n,1}')(q'_{n,1})+\dotsb+
F(r_{n,k_n}')(q'_{n,k_n})$ in $F(\fr/\fii_{n+1})$.
 Lifting the elements $q'_{n,i}$ to $q_{n,i}\in\PL(F)$ with
$\psi_{\fj_{r_{n,i}}}(q_{n,i})=q'_{n,i}$, we set
$p_{n+1}=p_n-r_{n,1}q_{n,1}-\dotsb-r_{n,k_n}q_{n,k_n}$, so
$\psi_{\fii_{n+1}}(p_{n+1})=0$.
 Starting from an element $p_0=p\in\PL(F)$ such that $\psi_{\fii}(p)=0$
and proceeding by induction in this way, we produce a family of elements
$r_{n,i}\in\fii$ converging to zero in the topology of $\fr$ and
a family of elements $q_{n,i}\in\PL(F)$ for which
$p=\pi_{\PL(F)}(\sum_{n,i}r_{n,i}q_{n,i})$.
 Hence $p\in\fii\tim\PL(F)$.

 Parts~(b\+-c): by the definition, $\PL(\CT(\fp))=\lim_{\fii}
\fp/(\fii\tim\fp)$, so we only have to prove that the natural
map to the projective limit of quotients
$\fp\to\lim_{\fii}\fp/(\fii\tim\fp)$ is surjective for any left
$\fr$-contramodule~$\fp$.
 Let $p=(p_{\fii})_{\fii\subseteq\fr}$ be an element of the projective
limit.
 Lift the elements $p_{\fii}\in\fp/(\fii\tim\fp)$ to some 
elements $p'_{\fii}\in\fp$.
 Let $\fii_0\supseteq\fii_1\supseteq\fii_2\supseteq\dotsb$ be
a chain of open right ideals indexed by the nonnegative integers
and cofinal in~$D$.
 For any $n\ge0$, the difference $p'_{\fii_{n+1}}-p'_{\fii_n}$ belongs
to $\fii_n\tim\fp$.
 Hence there is a set $X_n$, an $X_n$-indexed family of elements
$r_{n,x}\in\fii_n$ converging to zero in the topology of $\fr$, and
an $X_n$-indexed family of elements $q_{n,x}\in\fp$ such that
$p'_{\fii_{n+1}}-p'_{\fii_n}=\pi_{\fp}(\sum_{x\in X_n} r_{n,x}q_{n,x})$.
 Let $X=\coprod_n X_n$ be the disjoint union of the sets $X_n$;
then the $X$-indexed family of elements $r_{n,x}\in\fr$ converges
to zero in the topology of~$\fr$.
 Set $q=p_{\fii_0}'+\pi_{\fp}(\sum_{(n,x)\in X} r_{n,x}q_{n,x})\in\fp$.
 Then the difference $q-p'_{\fii_n}$ belongs to $\fii_n\tim\fp$
for every $n\ge0$, so the element $q$ is a preimage of
the projective limit element~$p$ in the contramodule~$\fp$.
\end{proof}

\begin{coro} \label{rex-separ-equivalence}
The restrictions of the functors\/ $\PL$ and $\CT$ are mutually inverse
equivalences between the full subcategory of colimit-preserving
functors\/ $\Rex(\Discrd\fr)\subseteq\Fun(\Discrd\fr)$ and
the full subcategory of separated contramodules\/
$\fr\lSepar\subseteq\fr\lContra$. \qed
\end{coro}

\begin{rem}
 Given a small category $\ca$ and a regular cardinal $\lambda$, denote
by $\Lex_\lambda(\ca)$ the category of all covariant functors
$\ca\to\Set$ preserving all $\lambda$-small limits existing in~$\ca$.
 The category $\Lex_\lambda(\ca)$ has all kinds of good properties;
in particular, it is a locally $\lambda$-presentable category and
a reflective subcategory of the category $\Set^\ca$ of all covariant
functors $\ca\to\Set$.
 Moreover, any locally $\lambda$-presentable category $\ck$ is
equivalent to $\Lex_\lambda(\ck_\lambda^{\op})$, where $\ck_\lambda$ is
the full subcategory of $\lambda$-presentable objects in~$\ck$
\cite[Theorem~1.46]{AR}.
 The construction of the reflector $\Set^\ca\to\Lex_\lambda(\ca)$
involves directed colimits over chains of length $\lambda$ in
the category $\Set$ \cite[Construction~1.37]{AR}.

 Furthermore, suppose that $\ca$ is an additive category.
 Then any functor $F:\ca\to\Set$ preserving finite products lifts
to an additive functor $F':\ca\to\Ab$ in a unique way.
 Limits in $\Ab$ agree with those in $\Set$, so the category
$\Lex_\lambda(\ca)$ can be equivalently defined as the category of
all covariant additive functors $\ca\to\Ab$ preserving
$\lambda$-small limits.
 When the category $\ca$ is abelian, the category $\Lex_\omega(\ca)$
of all left exact functors $\ca\to\Ab$ is abelian, too, and is
a localization of the abelian category $\Fadd(\ca)$ of all additive
functors $\ca\to\Ab$.
 So the inclusion $\Lex_\omega(\ca)\to\Fadd(\ca)$ is a left exact
functor, and the functor $\Fadd(\ca)\to\Lex_\omega(\ca)$ left adjoint
to the inclusion is exact.
 The construction of the latter functor involves filtered colimits
in the category~$\Ab$.
 This can be viewed as a kind of ``additive sheaf theory'':
$\Lex_\omega(\ca)$ is interpreted as the category of additive sheaves
on the category $\ca^{\op}$ endowed with the Grothendieck topology
with epimorphisms in $\ca^{\op}$ playing the role of generating
coverings.
 $\Fadd(\ca)$ is the category of additive presheaves, and
$\Fadd(\ca)\to\Lex_\omega(\ca)$ is the sheafification functor
\cite[Appendix~A]{Bu}.

 Unlike the limits, the colimits are not preserved by the forgetful
functor $\Ab\to\Set$, so colimit-preserving functors $\ca\to\Set$
are entirely unrelated to colimit-preserving functors $\ca\to\Ab$.
 Still, for an additive category $\ca$, any functor $\ca\to\Ab$
preserving finite coproducts is additive.
 For any small category $\ca$ one can consider the category
$\Rex_\lambda(\ca)$ of covariant functors $\ca\to\Ab$ preserving
$\lambda$-small colimits.
 However, the category $\Rex_\omega(\ca)$ does not have good properties.
 To be sure, let $R$ be an associative ring and $\cn=\Modr R$
the category of right $R$-modules.
 Then the category $\Rex(\cn)$ of colimit-preserving functors
$\cn\to\Ab$ is equivalent to the category of left $R$-modules
$R\lMod$.
 Furthermore, if $\ca\subset\cn$ is the category of finitely presented
right $R$-modules, then one has $\Rex_\omega(\ca)=\Rex(\cn)=R\lMod$,
so the category $\Rex_\omega(\ca)$ is abelian.
 But in general it is not, even for an abelian category~$\ca$.
 The explanation is that $\Rex_\omega(\ca)$ is a kind of ``additive
cosheaf category'', and cosheaves are known to be problematic.
 In particular, an attempt to dualize the sheaf theory would presume
constructing cosheafification in terms of filtered limits in $\Ab$,
and these are not exact functors (see the discussion
in~\cite[Section~2.5]{C} and the introduction to~\cite{P3}).

 Specifically, let $R$ be a Noetherian commutative ring, $I=(r)$
a principal ideal, and $\fr=\lim_m R/I^m$ the adic completion.
 Let $\cn=\Discrd\fr$ be the category of $(r)$-torsion $R$-modules
(see Example~\ref{contratensor-examples}(3)) and $\ca\subset\cn$
the full subcategory of finitely generated $(r)$-torsion $R$-modules.
 Then $\cn$ is a locally Noetherian abelian category and $\ca$ is
a Noetherian abelian category, but the category $\Rex_\omega(\ca)=
\Rex(\cn)$ is not abelian.
 Indeed, by Corollary~\ref{rex-separ-equivalence}, \ $\Rex(\cn)$
is equivalent to the category of separated $\fr$-contramodules
$\fr\lSepar$.
 A well-known counterexample demonstrates that the category $\fr\lSepar$
is not well-behaved.
 As a full subcategory in $\fr\lContra$, it is closed neither under
the cokernels of monomorphisms nor under extensions
(see~\cite[Example~2.5]{Si} or~\cite[Section~1.5]{P4}).
 All limits and colimits exist in $\fr\lSepar$, but the same
counterexample shows that the kernel of cokernel may differ from
the cokernel of kernel, so $\fr\lSepar$ is not abelian.
 The abelian category $\fr\lContra$ may be viewed as the ``corrected''
or ``right'' version of the cosheaf category $\Rex(\cn)$.

 Still, we have to consider colimit-preserving functors
$\Discrd\fr\to\Ab$ and separated contramodules in this section
for technical purposes.
\end{rem}

\begin{lemma} \label{derived-limit}
Let\/ $0\to H\to G\to F\to 0$ be a short exact sequence in\/
$\Fun(\Discrd\fr)$.
Assume that $H\in\Rex(\Discrd\fr)$.
Then the short sequence\/ $0\to\PL(H)\to\PL(G)\to\PL(F)\to0$ is
exact in\/ $\fr\lContra$.
\end{lemma}

\begin{proof}
 For any open right ideal $\fii$ in $\fr$, we have a short exact
sequence of abelian groups $0\to H(\fr/\fii)\to G(\fr/\fii)\to
F(\fr/\fii)\to0$.
 For any pair of embedded open right ideals $\fii'\subseteq\fii$,
the morphism $\fr/\fii'\to\fr/\fii$ is an epimorphism in
$\Discrd\fr$, so the induced morphism $H(\fr/\fii')\to H(\fr/\fii)$
is an epimorphism in~$\Ab$.
 Since the diagram $Q$ has a countable cofinal subdiagram, it follows
that the derived functor $\lim^1_{\fii} H(\fr/\fii)$ vanishes.
 Passing to the limit, we obtain a short exact sequence of abelian
groups $0\to\lim_{\fii}H(\fr/\fii)\to\lim_{\fii}G(\fr/\fii)\to
\lim_{\fii}F(\fr/\fii)\to0$, which means the desired short exact
sequence of left $\fr$-contramodules.
\end{proof}

By the definition, a left $\fr$-contramodule $\ff$ is flat if and only
if the functor $\CT(\ff)\in\Rex(\Discrd\fr)$ belongs to the full
subcategory $\Ex(\Discrd\fr)\subseteq\Rex(\Discrd\fr)$.
We will see below in Corollary~\ref{flat-is-separated} that any flat
$\fr$-contramodule is separated.

\begin{lemma} \label{main-lemma}
Let\/ $0\to \fq\to\fp\to\ff\to 0$ be a short exact sequence of left\/
$\fr$-contramodules.
Assume that the contramodule\/ $\fp$ is separated, while
the contramodule\/ $\ff$ is flat and separated.
Then for any discrete right\/ $\fr$-module $N$ the short sequence
of abelian groups\/ $0\to N\odot_{\fr}\fq\to N\odot_{\fr}\fp\to
N\odot_{\fr}\ff\to0$ is exact.
\end{lemma}

\begin{proof}
 Set $H=\CT(\fq)$, \ $G=\CT(\fp)$, and $F=\CT(\ff)$.
 The functor of contratensor product $N\odot_{\fr}{-}:\fr\lContra
\to\Ab$ is a left adjoint, so it preserves colimits.
 Hence the sequence of functors $H\to G\to F\to 0$ is exact in
$\Fun(\Discrd\fr)$.
 Denote by $H'$ the kernel of the morphism $G\to F$ in
$\Fun(\Discrd\fr)$; then the morphism $H\to G$ decomposes as
$H\to H'\to G$.
 The functor $F$ belongs to the subcategory $\Ex(\Discrd\fr)
\subseteq\Fun(\Discrd\fr)$ and the functor $G$ belongs to
$\Rex(\Discrd\fr)$, so by Lemma~\ref{exact-seq-functors}(a)
the functor $H'$ belongs to $\Rex(\Discrd\fr)$.

 Following Lemma~\ref{derived-limit}, we have a short exact sequence
of left $\fr$-contramodules $0\to\PL(H')\to\PL(G)\to\PL(F)\to0$.
 The $\fr$-contramodule $\fq$ is separated as a subcontramodule of
a separated $\fr$-contramodule $\fp$, so all the three adjunction
morphisms $\fq\to\PL(H)$, \ $\fp\to\PL(G)$, and $\ff\to\PL(F)$ are
isomorphisms of $\fr$-contramodules by
Lemma~\ref{adjunction-isomorphisms}(c).
 We can conclude that the morphism $\PL(H)\to\PL(H')$ is
an isomorphism of $\fr$-contramodules.
 Since both $H$ and $H'$ belong to $\Rex(\Discrd\fr)$, by
Corollary~\ref{rex-separ-equivalence} it follows that the morphism
of functors $H\to H'$ is an isomorphism and the short sequence
$0\to H\to G\to F\to 0$ is exact.
\end{proof}

\begin{coro} \label{flat-contra-kernels}
 The kernel of any surjective morphism of flat and separated left\/
$\fr$-contramodules is flat and separated.
\end{coro}

\begin{proof}
 Follows from Lemmas~\ref{main-lemma}
and~\ref{exact-seq-functors}(b).
\end{proof}

\begin{lemma} \label{proj-contramodules}
 Any projective left\/ $\fr$-contramodule is flat and separated.
\end{lemma}

\begin{proof}
It suffices to show that free left $\fr$-contramodules $\fr[[X]]$ are
flat and separated.
Indeed, one has $N\odot_{\fr}\fr=N$ for any discrete right $\fr$-module
$N$, so $\CT(\fr)$ is simply the forgetful functor $\Discrd\fr
\to\Ab$.
The functor $\CT$ is a left adjoint, so it preserves coproducts, hence
$N\odot_{\fr}\fr[[X]]=N^{(X)}$ is the coproduct of $X$ copies of
the underlying abelian group of a discrete right $\fr$-module~$N$.
This functor preserves all colimits and finite limits, thus
the left $\fr$-contramodule $\fr[[X]]$ is flat.
To prove that it is separated, one computes that $\fii\tim\fr[[X]]=
\fii[[X]]\subseteq\fr[[X]]$ for any open right ideal $\fii$ in $\fr$,
and $\bigcap_{\fii}\fii[[X]]=0$ because $\bigcap_{\fii}\fii=0$ in~$\fr$
(since $\fr$ is a separated topological ring).
\end{proof}

The next lemma only differs from Lemma~\ref{main-lemma} in that
the separatedness condition on the middle term of the exact sequence
of contramodules is dropped.

\begin{lemma} \label{gen-main-lemma}
Let\/ $0\to \fq\to\fp\to\ff\to 0$ be a short exact sequence of left\/
$\fr$-contramodules.
Assume the $\fr$-contramodule\/ $\ff$ is flat and separated.
Then the short sequence of functors\/ $0\to\CT(\fq)\to\CT(\fp)\to
\CT(\ff)\to0$ is exact in\/ $\Fun(\Discrd\fr)$.
\end{lemma}

\begin{proof}
Given a discrete right $\fr$-module $N$, the left derived functors 
of contratensor product $\Ctrtor^{\fr}_n(N,{-}):\fr\lContra\to\Ab$,
\ $n\ge0$ are defined as the homology groups
$$
 \Ctrtor^{\fr}_n(N,\fq)=H_n(N\odot_{\fr}\fp_\bullet)
$$
of the contratensor product of $N$ with a left resolution
$$
 \fp_\bullet=(\dotsb\longrightarrow\fp_2\longrightarrow\fp_1
 \longrightarrow\fp_0)\longrightarrow\fq
$$
of an $\fr$-contramodule $\fq$ by projective
$\fr$-contramodules~$\fp_n$.
Since the functor $\fq\mapsto N\odot_{\fr}\fq$ is right exact, one has
$\Ctrtor^{\fr}_0(N,\fq)=N\odot_{\fr}\fq$.
This definition of a left derived functor provides for any short exact
sequence of left $\fr$-contramodules $0\to\fq\to\fp\to\ff\to0$
a long exact sequence of abelian groups
\begin{multline*}
\dotsb\longrightarrow\Ctrtor^{\fr}_1(N,\fp)\longrightarrow
\Ctrtor^{\fr}_1(N,\ff) \\ \longrightarrow N\odot_{\fr}\fq
\longrightarrow N\odot_{\fr}\fp\longrightarrow N\odot_{\fr}\ff
\longrightarrow 0.
\end{multline*}
In particular, let $\ff$ be a flat and separated $\fr$-contramodule.
Taking $\fp$ to be a projective $\fr$-contramodule, we have
$\Ctrtor^{\fr}_1(N,\fp)=0$ by the definition.
By Lemmas~\ref{proj-contramodules} and~\ref{main-lemma}, the morphism
$N\odot_{\fr}\fq\to N\odot_{\fr}\fp$ is injective.
Hence $\Ctrtor^{\fr}_1(N,\ff)=0$, and it follows from the same long
exact sequence that the short sequence $0\to N\odot_{\fr}\fq\to
N\odot_{\fr}\fp\to N\odot_{\fr}\ff\to 0$ is exact for an arbitrary
$\fr$-contramodule~$\fp$.
\end{proof}

\begin{rem} \label{higher-ctrtor-vanishing}
Applying Lemma~\ref{main-lemma} together with
Corollary~\ref{flat-contra-kernels} and Lemma~\ref{proj-contramodules},
one shows easily that $\Ctrtor^{\fr}_n(N,\ff)=0$ for any flat and
separated left $\fr$-contramodule $\ff$, any discrete right
$\fr$-module $N$, and all $n\ge1$.
\end{rem}

\begin{rem}
The conventional derived functor $\Tor^R_*({-},{-})$ over an associative
ring $R$ can be computed by taking a flat projective resolution
of either of its arguments; the derived functors obtained by resolving
the first and the second argument of the tensor product functor
$\otimes_R$ are naturally isomorphic.
This important feature of the classical theory is absent in the case of
the $\Ctrtor$, as there may be no flat or projective objects in
the category of discrete right $\fr$-modules, in general.
So the functor $\Ctrtor^{\fr}_*({-},{-})$ can be only computed by
resolving the second (contramodule) argument of the functor of
contratensor product $\odot_{\fr}$, generally speaking.
This is the reason why the proofs of Lemmas~\ref{main-lemma}
and~\ref{gen-main-lemma}, as well as of
Remark~\ref{higher-ctrtor-vanishing}, are so much more indirect and
intricate than in the case of a discrete ring~$R$
(cf.~\cite[first Question in Section~5.3]{P}).
\end{rem}

\begin{coro}  \label{flat-contra-extensions}
The class of flat and separated\/ $\fr$-contramodules is closed under
extensions in the abelian category\/ $\fr\lContra$.
\end{coro}

\begin{proof}
Let $0\to\fq\to\fp\to\ff\to0$ be a short exact sequence of left
$\fr$-contramodules.
Assume that the contramodules $\fq$ and $\ff$ are flat and separated.
According to Lemma~\ref{gen-main-lemma}, the short sequence of
functors $0\to\CT(\fq)\to\CT(\fp)\to\CT(\ff)\to0$ is exact in
$\Fun(\Discrd\fr)$.
Applying Lemma~\ref{exact-seq-functors}(d), we can conclude that
$\fp$ is a flat $\fr$-contramodule.
To prove that $\fp$ is separated, consider the short sequence of
$\fr$-contramodules $0\to\PL(\CT(\fq))\to\PL(\CT(\fp))\to
\PL(\CT(\ff))\to0$, which is exact by Lemma~\ref{derived-limit}.
According to Lemma~\ref{adjunction-isomorphisms}(c), both the adjunction
morphisms $\fq\to\PL(\CT(\fq))$ and $\ff\to\PL(\CT(\ff))$ are
isomorphisms.
It follows that the adjunction morphism $\fp\to\PL(\CT(\fp))$ is
an isomorphism, too.
\end{proof}

The following result is a version of Nakayama's lemma for contramodules
(see~\cite[Lemma~1.3.1]{P2} for a formulation more similar to
the classical Nakayama's lemma).

\begin{lemma} \label{nakayama}
If\/ $\fp$ is a left\/ $\fr$-contramodule, and\/ $\CT(\fp)=0$,
then\/ $\fp=0$.
\end{lemma}

\begin{proof}
 We have to show that a left $\fr$-contramodule $\fp$ vanishes whenever
$\fp=\fii\tim\fp$ for every open right ideal $\fii$ in~$\fr$.
 Let $p\in\fp$ be an element, and let $\fii_1\supseteq\fii_2\supseteq
\dotsb$ be a chain of open right ideals, indexed by the integers,
forming a base of neighborhoods of zero in~$\fr$.
 For every $n\ge1$, choose a section $s_n$ of the surjective map
$\pi_{\fp}|_{\fii_n[[\fp]]}:\fii_n[[\fp]]\to\fp$; so
$s_n:\fp\to\fii_n[[\fp]]$ is a map of sets for which
$\pi_{\fp}s_n=\id_{\fp}$.

 Set $n_1=1$ and $p_1=s_{n_1}(p)\in\fii_{n_1}[[\fp]]$.
 Among the coefficients $r\in\fii_1$ appearing in the infinite formal
linear combination~$p_1$, there is only a finite set $R_2$ of those
which do not belong to~$\fii_2$.
 Let $n_2\ge2$ be an integer such that $r\fii_{n_2}\subseteq\fii_2$
for every $r\in R_2$.
 Put $p_2=\fii_{n_1}[[s_{n_2}]](p_1)$.
 Then one has $p_2\in\fii_{n_1}[[\fii_{n_2}[[\fp]]]]$, \
$p_1=\fii_{n_1}[[\pi_{\fp}]](p_2)$, and
$\phi_{\fp}(p_2)\in\fii_2[[\fp]]$.
 Put $q_2=p_2$.

 Among the coefficients $r\in\fii_2$ appearing in the infinite formal
linear combination $\phi_{\fp}(p_2)$, there is only a finite set $R_3$
of those which do not belong to~$\fii_3$.
 Let $n_3\ge3$ be an integer such that $r\fii_{n_3}\subseteq\fii_3$
for every $r\in R_3$.
 Put $p_3=\fii_{n_1}[[\fii_{n_2}[[s_{n_3}]]]](p_2)$.
 Then one has $p_3\in\fii_{n_1}[[\fii_{n_2}[[\fii_{n_3}[[\fp]]]]]]$,
\ $p_2=\fii_{n_1}[[\fii_{n_2}[[\pi_{\fp}]]]](p_3)$, \
$\phi_{\fii_{n_3}[[\fp]]}(p_3)\in\fii_2[[\fii_{n_3}[[\fp]]]]$, and
$\phi_{\fp}\phi_{\fii_{n_3}[[\fp]]}(p_3)\in\fii_3[[\fp]]$.
 Put $q_3=\phi_{\fii_{n_3}[[\fp]]}(p_3)$.
 
 Among the coefficients $r\in\fii_3$ appearing in the infinite formal
linear combination $\phi_{\fp}(q_3)\in\fii_3[[\fp]]$, there is only
a finite set $R_4$ of those which do not belong to~$\fii_4$.
 Let $n_4\ge4$ be an integer such that $r{\fii_{n_4}}\subseteq\fii_4$
for every $r\in R_4$.
 Put $p_4=\fii_{n_1}[[\fii_{n_2}[[\fii_{n_3}[[s_{n_4}]]]]]](p_3)$ and
$q_4=\phi_{\fii_{n_4}[[\fp]]}\phi_{\fii_{n_3}[[\fii_{n_4}[[\fp]]]]}(p_4)\in
\fii_3[[\fii_{n_4}[[\fp]]]]$, etc.

 For any set $X$, define inductively $\fr^{(0)}[[X]]=X$ and
$\fr^{(k)}[[X]]=\fr^{(k-1)}[[\fr[[X]]]]$ for $k\ge1$.
 Let $\phi^{(k)}_X:\fr^{(k)}[[X]]\to \fr[[X]]$ denote the iterated
monad multiplication map (opening the outer $k-1$ pairs of
parentheses or all but the innermost one pair, i.~e., all
the parentheses that can be opened for an arbitrary set~$X$).
 Proceeding as above, we construct a sequence of integers $n_k\ge k$
and a sequence of elements
$p_k\in\fii_{n_1}[[\dotsb[[\fii_{n_k}[[\fp]]]]\dotsb]]$ such that
$p_k=\fii_{n_1}[[\dotsb[[\fii_{n_{k-1}}[[s_{n_k}]]]]\dotsb]](p_{k-1})$
and the image of $p_k$ under the map $\phi^{(k)}_{\fp}:
\fr^{(k)}[[\fp]]\to\fr[[\fp]]$ belongs to $\fii_k[[\fp]]$.

 Then one has
$p_{k-1}=\fii_{n_1}[[\dotsb[[\fii_{n_{k-1}}[[\pi_{\fp}]]]]\dotsb]](p_k)$
and the image $q_k=\phi^{(k-1)}_{\fii_{n_k}[[\fp]]}(p_k)$ of the element
$p_k$, $k\ge2$ under the iterated monad multiplication map
$$
\phi_{\fii_{n_k}[[\fp]]}\phi_{\fii_{n_{k-1}}[[\fii_{n_k}[[\fp]]]]}
\dotsm\phi_{\fii_{n_3}[[\dotsb[[\fii_{n_k}[[\fp]]]]\dotsb]]}:
\fii_{n_1}[[\dotsb[[\fii_{n_k}[[\fp]]]]\dotsb]]\longrightarrow
\fr[[\fii_{n_k}[[X]]]]
$$
opening the outer $k-2$ pairs of parentheses (i.~e., all but
the innermost two pairs) belongs to $\fii_{k-1}[[\fii_{n_k}[[\fp]]]]$.
 The elements $q_k$ satisfy the equations 
$$
 \phi_{\fp}(q_k)=\phi^{(k)}_{\fp}(p_k)=\fii_k[[\pi_{\fp}]](q_{k+1}),
 \qquad k\ge2.
$$

 For any set $X$, we endow the abelian group $\fr[[X]]$ with
the topology of the limit of discrete groups $\lim_{\fii}\fr/\fii[X]$.
 For any map of sets $f:X\to Y$, the induced map $\fr[[f]]:
\fr[[X]]\to\fr[[Y]]$ is a continuous homomorphism of topological
groups.
 In particular, the set $\fr[[\fr[[\fp]]]]$ is endowed with
the topological group structure of $\fr[[X]]$, where $X=\fr[[\fp]]$
is viewed as an abstract set, while the set $\fr[[\fp]]$ is endowed with
the topological group structure of $\fr[[Y]]$ with the set $Y=\fp$.
 Then both the maps $\fr[[\pi_{\fp}]]$, $\phi_{\fp}:\fr[[\fr[[\fp]]]]
\rightrightarrows\fr[[\fp]]$ are continuous group homomorphisms.

 The sum $\sum_{k=2}^\infty q_k$ converges in the topology of
$\fr[[\fr[[\fp]]]]$, since $q_k\in\fii_{k-1}[[\fr[[\fp]]]]$.
 Therefore, we have
$$
 \fr[[\pi_{\fp}]]\Big(\sum\nolimits_{k=2}^\infty q_k\Big)
 - \phi_{\fp}\Big(\sum\nolimits_{k=2}^\infty q_k\Big) =
 \fr[[\pi_{\fp}]](q_2) = p_1,
$$
hence $p=\pi_{\fp}(p_1)=0$ by the contraassociativity equation.
\end{proof}

\begin{coro} \label{flat-is-separated}
 Any flat left\/ $\fr$-contramodule is separated.
\end{coro}

\begin{proof}
 Let $\fp$ be a flat left $\fr$-contramodule.
 Set $\ff=\PL(\CT(\fp))$.
 Then $\CT(\ff)=\CT(\fp)$ by Lemma~\ref{adjunction-isomorphisms}(a),
so the $\fr$-contramodule $\ff$ is flat.
 Following Lemma~\ref{pl-ct-adjunction}(b), $\ff$ is also separated.
 According to Lemma~\ref{adjunction-isomorphisms}(b), the adjunction
morphism $\fp\to\ff$ is surjective.
 Let $\fq$ denote its kernel.
 Following Lemma~\ref{gen-main-lemma}, we have a short exact sequence
$0\to\CT(\fq)\to\CT(\fp)\to\CT(\ff)\to0$ in $\Fun(\Discrd\fr)$.
 Hence $\CT(\fq)=0$, and by Lemma~\ref{nakayama} we can conclude
that $\fq=0$.
 Thus $\fp=\ff$ and $\fp$ is separated.
\end{proof}

\begin{lemma} \label{colim-exact-on-flats}
Let\/ $(0\to\fq_i\to\fp_i\to\ff_i\to0)_{i\in\cc}$ be a filtered diagram
of short exact sequences in\/ $\fr\lContra$.
Assume that every\/ $\fr$-contramodule\/ $\fq_i$, $\fp_i$, and\/
$\ff_i$ is flat.
Then the short sequence of colimits\/ $0\to\colim_i\fq_i\to
\colim_i\fp_i\to\colim_i\ff_i\to0$ is exact in\/ $\fr\lContra$.
\end{lemma}

\begin{proof}
 Following Corollary~\ref{flat-is-separated}, the $\fr$-contramodules
$\ff_i$ are separated.
 Hence the short sequence of functors $0\to\CT(\fq_i)\to\CT(\fp_i)\to
\CT(\ff_i)\to0$ is exact in $\Fun(\Discrd\fr)$ for every~$i$ by
Lemma~\ref{gen-main-lemma}.
 Passing to the colimit, we obtain a short exact sequence of
functors
$$
0\longrightarrow\CT(\colim_i\fq_i)\longrightarrow\CT(\colim_i\fp_i)
\longrightarrow\CT(\colim_i\ff_i)\longrightarrow0.
$$
 By Lemma~\ref{derived-limit}, we have a short exact sequence of
left $\fr$-contramodules
$$
0\longrightarrow\PL(\CT(\colim_i\fq_i))\longrightarrow
\PL(\CT(\colim_i\fp_i))\longrightarrow\PL(\CT(\colim_i\ff_i))
\longrightarrow0.
$$
 According to Lemma~\ref{flats-colimit-closed}, the $\fr$-contramodules
$\colim_i\fq_i$, $\colim_i\fp_i$, and $\colim_i\ff_i$ are flat, and
following Corollary~\ref{flat-is-separated} again, they are also
separated.
 It remains to apply Lemma~\ref{adjunction-isomorphisms}(c) in order
to obtain the desired short exact sequence.
\end{proof}

\begin{propo} \label{colim1-vanishes-on-flats}
Let\/ $(0\to\fq_i\to\fp_i\to\ff_i\to0)_{i\in\cc}$ be a filtered diagram
of short exact sequences in\/ $\fr\lContra$.
Assume that every\/ $\fr$-contramodule\/ $\ff_i$ is flat.
Then the short sequence of colimits
$$
0\longrightarrow\colim_i\fq_i\longrightarrow\colim_i\fp_i
\longrightarrow\colim_i\ff_i\longrightarrow0
$$
is exact in\/ $\fr\lContra$.
\end{propo}

\begin{proof}
 For any small category $\cc$ and any cocomplete abelian category $\ck$
with enough projective objects, the category of functors $\ck^\cc$ has
enough projective objects.
 In fact, these are the direct summands of coproducts of functors of
the form $P^{(\cc(A,-))}$, where $P$ is a projective object in $\ck$
and $A$ an object in~$\cc$. 
 Furthermore, for any projective object $Q$ in $\ck^\cc$ and any object
$C$ in $\cc$ the object $Q(C)$ is projective in~$\ck$.

 One constructs the left derived functors of colimit
${}_n\!\colim_{\cc}:\ck^\cc\to\ck$, $n\ge0$ by choosing
a left projective resolution
$$
 Q_\bullet=(\dotsb\longrightarrow Q_2\longrightarrow Q_1
 \longrightarrow Q_0)\longrightarrow F
$$
in the abelian category of functors $\ck^\cc$ for a given functor $F:
\cc\to\ck$, computing the colimit $\colim_{\cc}Q_n$ of every
functor $Q_n:\cc\to\ck$, and passing to the homology objects of
the complex formed by such colimits in the category~$\ck$,
$$
 {}_n\!\colim_{\cc}(F)=H_n(\colim_{\cc}Q_\bullet).
$$
 Since the functor $\colim_{\cc}:\ck^\cc\to\ck$ is right exact, one
has ${}_0\!\colim_{\cc}F=\colim_{\cc}F$.
 This definition of a left derived functor provides for any short
exact sequence of functors $0\to H\to G\to F\to 0$ in $\ck^\cc$
a long exact sequence
$$
 \dotsb\longrightarrow{}_1\!\colim_{\cc}G\longrightarrow
 {}_1\!\colim_{\cc}F\longrightarrow\colim_{\cc}H\longrightarrow
 \colim_{\cc}G\longrightarrow\colim_{\cc}F\longrightarrow0
$$
in the category~$\ck$.

 Now we specialize to the case of a filtered category $\cc$ of indices
$i$ and the category of contramodules $\ck=\fr\lContra$.
 Let $(\ff_i)_{i\in\cc}$ be a diagram of flat $\fr$-contramodules.
 First we choose a diagram $(\fp_i)_{i\in\cc}$ in such a way that
it is a projective object in $\fr\lContra^\cc$ and there is
an epimorphism $(\fp_i)\to(\ff_i)$ in $\fr\lContra^\cc$ with
a kernel $(\fq_i)_{i\in\cc}$.
 Then, by the definition, ${}_1\!\colim_i\fp_i=0$.
 According to the above discussion, every $\fr$-contramodule $\fp_i$
is projective; hence, by Lemma~\ref{proj-contramodules}, \ $\fp_i$
is flat.
 Due to Corollaries~\ref{flat-contra-kernels}
and~\ref{flat-is-separated} we know that $\fq_i$ is flat for
every $i$, too.
 Following Lemma~\ref{colim-exact-on-flats}, the short sequence
$0\to\colim_i\fq_i\to\colim_i\fp_i\to\colim_i\ff_i\to0$ is exact.
 Hence we can conclude from the long exact sequence 
$$
 \dotsb\longrightarrow{}_1\!\colim_i\fp_i\longrightarrow
 {}_1\!\colim_i\ff_i\longrightarrow\colim_i\fq_i\longrightarrow
 \colim_i\fp_i\longrightarrow\colim_i\ff_i\longrightarrow0
$$
that ${}_1\!\colim_i\ff_i=0$.
 Finally, for an arbitrary diagram $(\fp_i)\in\fr\lContra^\cc$
and a short exact sequence of diagrams
$(0\to\fq_i\to\fp_i\to\ff_i\to0)_{i\in\cc}$ the same long exact
sequence of derived colimits shows that the short sequence
$0\to\colim_i\fq_i\to\colim_i\fp_i\to\colim_i\ff_i\to0$ is exact.
\end{proof}

\begin{rem}
Similarly one can show that ${}_n\!\colim_i\ff_i=0$ for any filtered
diagram of flat left $\fr$-contramodules $(\ff_i)_{i\in\cc}$ and all
$n\ge1$.
\end{rem}

\section{Cotorsion theories in Contramodule categories}
\label{cotorsion-in-contramodules}

 We keep the assumptions of Section~\ref{flat-contramodules-secn},
i.~e., $\fr$ is a complete and separated, associative and unital
topological ring with a countable base of neighborhoods of zero
formed by open right ideals.
The following corollary summarizes the results of
Section~\ref{flat-contramodules-secn}.

\begin{coro} \label{flat-contramodules-summary}
The full subcategory of flat\/ $\fr$-contramodules in\/
$\fr\lContra$ is \par
\textup{(a)} closed under the kernels of epimorphisms; \par
\textup{(b)} closed under extensions; \par
\textup{(c)} strongly closed under directed colimits.
\end{coro}

\begin{proof}
Taking into account Corollary~\ref{flat-is-separated},
part~(a) is provided by Corollary~\ref{flat-contra-kernels}
and part~(b) by Corollary~\ref{flat-contra-extensions}.
Part~(c) follows from Proposition~\ref{colim1-vanishes-on-flats}
and Lemma~\ref{flats-colimit-closed}.
\end{proof}

\begin{propo} \label{flat-pure-closed}
 The full subcategory of flat\/ $\fr$-contramodules is closed under\/
$\aleph_1$-pure subobjects and\/ $\aleph_1$-pure quotients in\/
$\fr\lContra$.
\end{propo}

\begin{proof}
The argument is based on the construction of the derived functor
$\Ctrtor^{\fr}_*({-},{-})$ from the proof of Lemma~\ref{gen-main-lemma}.
Let $\fp$ be a flat left $\fr$-contramodule, $\fq\to\fp$ 
an $\aleph_1$-pure monomorphism, and $\fp\to\ff=\fp/\fq$
the corresponding $\aleph_1$-pure epimorphism.
According to Proposition~\ref{lambda-pure-contratensor}, the morphism
of abelian groups $N\odot_{\fr}\fq\to N\odot_{\fr}\fp$ is
a monomorphism for any discrete right $\fr$-module~$N$.
Following the proof of Lemma~\ref{gen-main-lemma}, we have
$\Ctrtor^{\fr}_1(N,\fp)=0$ for a flat $\fr$-contramodule $\fp$, thus
we can conclude from the long exact sequence that
$\Ctrtor^{\fr}_1(N,\ff)=0$.

Since projective left $\fr$-contramodules are flat by
Lemma~\ref{proj-contramodules}, taking the contratensor product of
a short exact sequence of discrete right $\fr$-modules
$0\to L\to M\to N\to 0$ with a complex of projective left
$\fr$-contramodules $\fp_\bullet$ produces a short exact sequence of
complexes of abelian groups $0\to L\odot_{\fr}\fp_\bullet\to
M\odot_{\fr}\fp_\bullet\to N\odot_{\fr}\fp_\bullet\to0$.
Hence for any left $\fr$-contramodule $\ff$ there is a long exact
sequence
$$
\dotsb\longrightarrow\Ctrtor^{\fr}_1(N,\ff)\longrightarrow
L\odot_{\fr}\ff\longrightarrow M\odot_{\fr}\ff\longrightarrow
N\odot_{\fr}\ff\longrightarrow 0,
$$
and it follows that $\ff$ is flat whenever one has
$\Ctrtor^{\fr}_1(N,\ff)=0$ for every discrete right $\fr$-module $N$.
Finally, the $\fr$-contramodule $\fq$ is flat by 
Corollary~\ref{flat-contramodules-summary}(a).
\end{proof}

\begin{defi}
 A left\/ $\fr$-contramodule\/ $\fg$ is called \textit{cotorsion} if
$\Ext^1_{\fr\lContra}(\ff,\fg)=0$ for any flat left
$\fr$-contramodule~$\ff$.
\end{defi}

We will denote the class of all flat left $\fr$-contramodules by
$\fr\lFlat$ and the class of all cotorsion left $\fr$-contramodules
by $\fr\lCotors\subseteq\fr\lContra$.
By the definition, the class $\fr\lCotors$ is closed under extensions
and products in $\fr\lContra$.
Since the class of $\fr\lFlat$ contains projective
$\fr$-contramodules and is closed under the kernels of epimorphisms
(Lemma~\ref{proj-contramodules} and
Corollary~\ref{flat-contramodules-summary}(a)), one has
$\Ext_{\fr\lContra}^n(\ff,\fg)$ for any flat $\fr$-contramodule
$\ff$, any cotorsion $\fr$-contramodule $\fg$, and all $n\ge1$.
It follows that the class of cotorsion $\fr$-contramodules is closed
under the cokernels of monomorphisms (see, e.~g.,
\cite[Lemma~6.17]{S0}).

\begin{exam}
For any discrete right $\fr$-module $N$ and left $\fr$-contramodule
$\fp$, there is a natural isomorphism of abelian groups
$$
 \Ext^n_{\fr\lContra}(\fp,\Ab(N,\rationals/\integers))=
 \Ab(\Ctrtor_n^{\fr}(N,\fp),\rationals/\integers),
 \qquad n\ge0,
$$
as one can see by computing both the $\Ext$ and the $\Ctrtor$ in terms
of the same left projective resolution of an $\fr$-contramodule~$\fp$.
Consequently, the left $\fr$-contramodule
$\Ab(N,\rationals/\integers)$ is cotorsion for any discrete
right $\fr$-module~$N$.
\end{exam}

\begin{coro} \label{flat-accessible}
 The category of flat left\/ $\fr$-contramodules is accessible.
\end{coro}

\begin{proof}
 By Lemma~\ref{flats-colimit-closed}, \ $\fr\lFlat$ is accessibly
embedded into $\fr\lContra$, and by
Proposition~\ref{flat-pure-closed}, \ $\fr\lFlat$ is closed under
$\aleph_1$-pure subobjects.
 According to~\cite[Corollary~2.36]{AR}, it follows that
$\fr\lFlat$ is accessible.
\end{proof}

\begin{coro} \label{flat-deconstructible}
 The class of flat\/ $\fr$-contramodules is deconstructible in\/
$\fr\lContra$.
\end{coro}

\begin{proof} 
The class $\fr\lFlat$ is closed under transfinitely iterated
extensions in $\fr\lContra$, since it is closed under extensions
and filtered colimits (Lemmas~\ref{flat-contramodules-summary}(b)
and~\ref{flats-colimit-closed}).
It is closed under $\aleph_1$-pure subobjects and $\aleph_1$-pure
quotients by Proposition~\ref{flat-pure-closed}.
The class of all monomorphisms with flat cokernels is closed under
filtered colimits in the category of all morphisms of
$\fr$-contramodules $\fr\lContra^{\to}$ by
Proposition~\ref{colim1-vanishes-on-flats}.
Hence the desired assertion is provided by
Lemma~\ref{deconstructible}.
\end{proof}

\begin{coro} \label{by-set-of-contramodules}
Let $(\cf,\cc)$ be the cotorsion theory generated by a set of objects
$\cs$ in\/ $\fr\lContra$.
Then any left\/ $\fr$-contramodule has a special $\cf$-precover, and
any left\/ $\fr$-contramodule that can be embedded as a subcontramodule
into an\/ $\fr$-contramodule from $\cc$ has a special $\cc$-preenvelope.

Furthermore, an\/ $\fr$-contramodule belongs to $\cf$ if and only if it
is a direct summand of an\/ $\fr$-contramodule from $\filt(\cs')$, where
the set $\cs'=\cs\cup\{\fr\}$ consists of the contramodules from $\cs$
and the free left\/ $\fr$-contramodule\/~$\fr$.
\end{coro}

\begin{proof}
Clearly, one has $\cs^\perp=\cs'{}^\perp$, so the cotorsion theory
generated by $\cs$ in $\fr\lContra$ coincides with the one
generated by~$\cs'$.
The class $\filt(\cs')$ contains all the free left $\fr$-contramodules
$\fr[[X]]$, and any left $\fr$-contramodule $\fp$ is a quotient
contramodule of $\fr[[X]]$ for a certain set~$X$.
All the assertions of the corollary now follow from the related
assertions of Theorem~\ref{precise-et-loc-pres}.
\end{proof}

\begin{coro} \label{flat-cotorsion-theory}
The pair of full subcategories\/ $(\fr\lFlat$, $\fr\lCotors)$
forms a complete cotorsion theory in\/ $\fr\lContra$.
\end{coro}

\begin{proof}[First proof]
 By Corollary~\ref{flat-deconstructible}, there is a set of flat
$\fr$-contramodules $\cs$ such that $\fr\lFlat=\filt(\cs)$.
 Replacing if needed $\cs$ by $\cs'=\cs\cup\{\fr\}$, we may assume
that the set $\cs$ contains the free $\fr$-contramodule~$\fr$.
 By Lemma~\ref{eklof-lemma}, we have $\fr\lCotors=\cs^\perp$.
 Denoting by $(\cf,\cc)$ the cotorsion theory generated by $\cs$
in $\fr\lContra$, we see that $\cc=\fr\lCotors$ and,
following Corollary~\ref{by-set-of-contramodules}, the class
$\cf$ consists precisely of the direct summands of objects from
$\filt(\cs)$.
 Hence $\cf=\fr\lFlat$.
 Now we know from Corollary~\ref{by-set-of-contramodules} that
any left $\fr$-contramodule has a special flat precover, and it
only remains to prove that it has a special cotorsion preenvelope.

 Any left $\fr$-contramodule $\fq$ is endowed with a natural
$\fr$-contramodule morphism into a cotorsion $\fr$-contramodule
$$
 c_{\fq}:\fq\longrightarrow\prod\nolimits_{\fii,f}
 \Ab(\fr/\fii,\rationals/\integers),
$$
where $\fii$ runs over all the open right ideals $\fii\subseteq\fr$
and $f\in\Ab(\fq/(\fii\tim\fq),\rationals/\integers)$.
The morphism $c_{\fq}$ assigns to any elements $q\in\fq$ and
$r\in\fr/\fii$, and an abelian group homomorphism $f$, the element
$f(rq)\in\rationals/\integers$.
The kernel of $c_{\fq}$ is precisely the $\fr$-subcontramodule
$\cap_{\fii}(\fii\tim\fq)\subseteq\fq$, i.~e., the kernel of
the adjunction morphism $\fq\to\PL(\CT(\fq))$.
Applying Corollary~\ref{by-set-of-contramodules} again, we can conclude
that any separated left $\fr$-contramodule has a special cotorsion
preenvelope.

Given an arbitrary left $\fr$-contramodule $\fq$, let $\ff\to\fq$ be
one of its special flat precovers.
The $\fr$-contramodule $\ff$ is separated by
Corollary~\ref{flat-is-separated}, hence it has a special cotorsion
preenvelope $\ff\to\fg$.
Let $\fk$ be the kernel of the epimorphism of
$\fr$-contramodules $\ff\to\fq$.
Then the composition $\fk\to\ff\to\fg$ is a monomorphism of
cotorsion $\fr$-contramodules, so its cokernel $\fg/\fk$ is
also cotorsion.
Now the natural morphism $\fq\to\fg/\fk$ is a special cotorsion
preenvelope of~$\fq$. (This is the argument dual
to~\cite[Lemma~3.1.3]{Xu};
cf.~\cite[proof of Lemma~2 in Section~9.1]{P}.)
\end{proof}

\begin{proof}[Second proof]
 We have already seen in the first paragraph of the first proof that
$(\fr\lFlat$, $\fr\lCotors)$ is a cotorsion theory generated
by a set of objects $\cs$ in $\fr\lContra$ and that any left
$\fr$-contramodule has a special flat precover.
 To show that it has a special cotorsion preenvelope, we apply
Proposition~\ref{et-sc-dc}.
 Indeed, the set $\cs$ is contained in the class of all flat left
$\fr$-contramodules, which is closed under extensions and strongly
closed under directed colimits in $\fr\lContra$ by
Corollary~\ref{flat-contramodules-summary}(b\+-c).
\end{proof}

\begin{proof}[Third proof]
 Deduce existence of special flat precovers of $\fr$-contramodules
from existence of their flat covers, provable by the Bican--El~Bashir
approach as in the first proof of the next Corollary~\ref{flat-covers}.
 To prove existence of special cotorsion preenvelopes of
subcontramodules of cotorsion contramodules, argue as in the proof of
Salce's lemma~\cite[Lemma~2.3]{Sa} (cf.\ the dual version in
the second half of the proof of~\cite[Theorem~10]{ET}).
 Special cotorsion preenvelopes of arbitrary $\fr$-contramodules
can be then obtained as in the second and third paragraphs of
the first proof above.
\end{proof}

\begin{coro} \label{flat-covers}
Any left\/ $\fr$-contramodule has a flat cover.
\end{coro}

\begin{proof}[First proof (Bican--El~Bashir approach to flat covers)]
Following Lemma~\ref{flats-colimit-closed}, the full subcategory
$\fr\lFlat$ is closed under directed colimits in
$\fr\lContra$.
Being an additive subcategory, it is consequently also closed
under coproducts.
By Corollary~\ref{flat-accessible}, \ $\fr\lFlat$ is accessible.
Thus Corollary~\ref{cor2.7} allows to conclude that $\fr\lFlat$
is stably weakly coreflective in $\fr\lContra$.
\end{proof}

\begin{proof}[Second proof (Eklof--Trlifaj approach to flat covers)]
Since $(\fr\lFlat$, $\fr\lCotors)$ is a complete cotorsion
theory by Corollary~\ref{flat-cotorsion-theory} and
$\fr\lFlat$ is closed under directed colimits by
Lemma~\ref{flats-colimit-closed}, we can conclude that $\fr\lFlat$
is stably weakly coreflective by Corollary~\ref{cor3.6}.
\end{proof}

\begin{coro}
Any left\/ $\fr$-contramodule has a cotorsion envelope.
\end{coro}

\begin{proof}
Since $\fr\lFlat$ is strongly closed under directed colimits in
$\fr\lContra$ by Corollary~\ref{flat-contramodules-summary}(c)
and $(\fr\lFlat$, $\fr\lCotors)$ is a complete cotorsion theory
by Corollary~\ref{flat-cotorsion-theory}, \ $\fr\lCotors$ is stably
weakly reflective in $\fr\lContra$ by
Corollary~\ref{cor3.9} or~\ref{precise-envelope}.
\end{proof}

\begin{coro} \label{set-of-flat}
Any cotorsion theory generated by a set of flat\/ $\fr$-contramodules
is complete in\/ $\fr\lContra$.
\end{coro}

\begin{proof}[First proof]
 Let $(\cf,\cc)$ be the cotorsion theory generated by a set of flat
$\fr$-contramodules~$\cs$.
 Following Corollary~\ref{by-set-of-contramodules}, any left
$\fr$-contramodule has a special $\cf$-precover, and any left
$\fr$-contramodule embeddable into an object of $\cc$ has a special
$\cc$-preenvelope.
 Now we have $\cs\subseteq\fr\lFlat$, hence $\cc\supseteq
\fr\lContra$, and it remains to recall that any left
$\fr$-contramodule is a subcontramodule of a cotorsion
$\fr$-contramodule by Corollary~\ref{flat-cotorsion-theory}.
\end{proof}

\begin{proof}[Second proof]
 To prove that any left $\fr$-contramodule has a special
$\cc$-preenvelope, notice that $\cs\subseteq\fr\lFlat$ and
the class $\fr\lFlat$ is strongly closed under directed colimits
by Corollary~\ref{flat-contramodules-summary}(c), so
Proposition~\ref{et-sc-dc} applies.
\end{proof}

\begin{exams}
(1) Let $\fr$ be a complete and separated topological associative ring
with a countable base of neighborhoods of zero $B$ formed by open
two-sided ideals (cf.\ Example~\ref{contratensor-examples}(2)).
 Let $\kappa$ be an uncountable regular cardinal.
 Since the fully faithful inclusion $\fr/\fii\lMod\to\fr\lContra$
preserves $\kappa$-filtered colimits, the reduction functor
$\fp\mapsto\fp/(\fii\tim\fp)$ takes $\kappa$-presentable
$\fr$-contramodules to $\kappa$-presentable $\fr/\fii$-modules.
 Conversely, using Nakayama's lemma~\ref{nakayama} one easily shows
that any $\fr$-contramodule $\fp$ whose reductions $\fp/(\fii\tim\fp)$
are $\kappa$-generated $\fr/\fii$-modules for all $\fii\in B$ is
$\kappa$-generated.
 Applying Lemma~\ref{main-lemma}, one deduces the claim that any flat
$\fr$-contramodule $\ff$ whose reductions $\ff/(\fii\tim\ff)$ are
$\kappa$-presentable for all $\fii\in B$ is $\kappa$-presentable.
 So a flat $\fr$-contramodule is $\kappa$-presentable if and only if
its reductions by open ideals $\fii\subseteq\fr$ are
$\kappa$-presentable $\fr/\fii$-modules for all $\fii\in B$.

\medskip

(2) Suppose that for every two-sided open ideal $\fii\subseteq\fr$,
\ $\fii\in B$ we are given a deconstructible class of flat left
$\fr/\fii$-modules $\cf_{\fii}\subseteq\fr/\fii\lFlat$ such that
for every pair of embedded ideals $\fii\subseteq\fii'\subseteq\fr$,
\ $\fii$, $\fii'\in B$, and any module $F$ from $\cf_{\fii'}$
the $\fr/\fii$-module $\fr/\fii\otimes_{\fr/\fii'}F$ belongs to
$\cf_{\fii}$.
 Then the class $\cf\subseteq\fr\lFlat$ of all left
$\fr$-contramodules $\ff$ such that $\ff/(\fii\tim\ff)\in\cf_{\fii}$
for every $\fii\in B$ is deconstructible.

 Specifically, let $\kappa$ be an uncountable regular cardinal such that
$\cf_{\fii}=\filt(\cs_{\fii})$ and $\cs_{\fii}$ is a set of
$\kappa$-presentable $\fr/\fii$-modules for every $\fii\in B$.
 Let $\cs$ consist of all left $\fr$-contramodules $\fs\in\cf$ whose
reductions $\fs/(\fii\tim\fs)$ are $\kappa$-presentable
for all $\fii\in B$ (i.~e., of all $\kappa$-presentable
contramodules from~$\cf$).
 We claim that $\cf=\filt(\cs)$.

 Indeed, let $\ff$ be a contramodule from $\cf$, and let $\fii_0
\supseteq\fii_1\supseteq\fii_2\supseteq\dotsb\in B$ be a chain of
ideals indexed by the nonnegative integers and cofinal in~$B$.
 According to Hill's lemma~\cite{ST}, for every $n\ge0$ there exists
a family of submodules $\cm_n$ in the $R/\fii_n$-module $F_n=
\ff/(\fii_n\tim\ff)$ closed under arbitrary sums and intersections
in $F_n$ and having the properties (H3) and (H4)
from~\cite[Theorem~6]{ST}.
 These state that (H3) for any $N$ and $P\in\cm_n$ with $N\subseteq P$
one has $P/N\in\cf_{\fii_n}$, and (H4) for any subset $X\subseteq F_n$
of cardinality less than $\kappa$ and $N\in\cm_n$ there exists
$P\in\cm_n$ such that $N\cup X\subseteq P$ and $P/N$ is
$\kappa$-presentable.

 Let $\lambda$ be the cardinality of~$\ff$.
 For any nonnegative integers $m\le n$, we have the reduction map
$f_{n,m}:F_n\to\fr/\fii_m\otimes_{\fr/\fii_n}F_n=F_m$; these maps form
a commutative diagram and the reduction maps $f_n:\ff\to F_n$ form
a limit cone (since flat contramodules are separated).
 We construct smooth chains of submodules $(F_{n,i})_{i\leq\lambda}$ in
the $R/\fii_n$-modules $F_n$ such that $F_{n,i}\in\cm_n$, \
$F_{n,i+1}/F_{n,i}$ is $\kappa$-presentable, and
$F_{m,i}=f_{n,m}(F_{n,i})$ for $m<n$.
 Set $F_{n,0}=0$.

 On a successor step $i+1$, if $F_{n,i}=F_n$ for all $n$, we set
$F_{n,i+1}=F_{n,i}$.
 Otherwise, let $m$ be the minimal nonnegative integer such that
$F_{m,i}\varsubsetneq F_m$.
 Choose $G_{m,1}\in\cm_m$ such that $F_{m,i}\varsubsetneq G_{m,1}$
and $G_{m,1}/F_{m,i}$ is $\kappa$-presentable.
 Proceeding by induction in $n\ge m$, choose $G_{n,1}\in\cm_n$
such that $G_{n,1}/F_{n,i}$ is $\kappa$-presentable and
$f_{n+1,n}(G_{n+1,1})\supseteq G_{n,1}$.
 Choose $G_{m,2}\in\cm_m$ such that $f_{p,m}(G_{p,1})\subseteq G_{m,2}$
for all $p\ge m$ and $G_{m,2}/F_{m,i}$ is $\kappa$-presentable.
 Proceeding by induction in $k\ge1$, we choose submodules
$G_{n,k}\in\cm_n$, \ $n\ge m$ such that $f_{n+1,n}(G_{n+1,k})\supseteq
G_{n,k}$, \ $f_{p,n}(G_{p,k})\subseteq G_{n,k+1}$ for all $p\ge n\ge m$,
and $G_{n,k}/F_{n,i}$ is $\kappa$-presentable.
 Now we can set $F_{n,i+1}=F_{n,i}$ for $n<m$ and
$F_{n,i+1}=\bigcup_k G_{n,k}$ for all $n\ge m$.
 Clearly $F_{n,\lambda}=F_n$ for all $n\ge0$.

 Setting $\ff_i=\lim_n F_{n,i}$, we obtain a smooth chain of left
$\fr$-contramodules $(\ff_i\to\ff_j)_{i<j\le\lambda}$ such that
$\ff_0=0$, \ $\ff_\lambda=\ff$, and the morphisms $\ff_i\to\ff_{i+1}$
are $\cs$-monomorphisms.
 We have shown that $\cf=\filt(\cs)$, hence according to
Eklof's lemma~\ref{eklof-lemma} and Corollary~\ref{set-of-flat}
the cotorsion theory $({}^\perp(\cf^\perp)$, $\cf^\perp)$ is complete
in $\fr\lContra$.
 Following Corollary~\ref{by-set-of-contramodules}, \
${}^\perp(\cf^\perp)$ is the class of all direct summands of
the contramodules from $\filt(\cs\cup\{\fr\})$.

\medskip

(3) Now let $\fr$ be a complete and separated topological commutative
ring with a countable base of neighborhoods of zero.
 For any commutative ring $R$, we denote by $(R\lVfl$, $R\lCtadj)$
the cotorsion theory generated by the set of $R$-modules $R[s^{-1}]$,
\ $s\in R$.
 Modules from $R\lVfl$ are called \textit{very flat} and modules
from $R\lCtadj$ are called \textit{contraadjusted}; so
$R\lVfl\subseteq R\lFlat$ and $R\lCtadj\supseteq R\lCotors$
(see~\cite[Section~1.1]{P3}).

 An $\fr$-contramodule $\ff$ is called \textit{very flat}
(cf.~\cite[Sections~C.3 and~D.4]{P3}) if the $\fr/\fii$-module
$\ff/(\fii\tim\ff)$ is very flat for every open ideal
$\fii\subseteq\fr$.
 An $\fr$-contramodule $\fg$ is \textit{contraadjusted} if
$\Ext^1_{\fr\lContra}(\ff,\fg)=0$ for any very flat
$\fr$-contramodule~$\ff$.
 The class of very flat $\fr$-contramodules is denoted by $\fr\lVfl$
and the class of contraadjusted $\fr$-contramodules by $\fr\lCtadj$.

 Let $\cs_R$ denote the (representative) set of all
$\aleph_1$-presentable very flat $R$-modules.
 Following~\cite[Lemma~9]{ST}, we have $R\lVfl=\filt(\cs_R)$.
 Hence, according to~(2), \ $\fr\lVfl=\filt(\cs_{\fr})$, where
$\cs_\fr$ consists of all $\fr$-contramodules $\ff$ such that
$\ff/(\fii\tim\ff)\in\cs_{\fr/\fii}$ for all open ideals
$\fii\subseteq\fr$.
 Since the class $\fr\lVfl$ contains the free $\fr$-contramodules
and is closed under direct summands, we can conclude that
$(\fr\lVfl$, $\fr\lCtadj)$ is a complete cotorsion theory in
the category $\fr\lContra$.
\end{exams}

\end{document}